\documentclass[12pt]{amsart}
\usepackage{bbm}
\usepackage[pdftex]{graphicx}
\usepackage{amscd, amssymb, dsfont, upgreek, mathrsfs}
\usepackage[dvipsnames]{xcolor}
\input{xy}
\xyoption{all}
\newtheorem{Thm}{Theorem}

\newtheorem{Prop}[Thm]{Proposition}
\newtheorem{Def}[Thm]{Definition}
\newtheorem{Def/Thm}[Thm]{Definition/Theorem}
\newtheorem{Cor}[Thm]{Corollary}
\newtheorem{Lemma}[Thm]{Lemma}

\theoremstyle{remark}

\newcommand{\gG}{{\mathbf{g}}}

\newcommand{\lann}{\langle\langle}
\newcommand{\rann}{\rangle\rangle}

\newcommand{\PP }{{\mathbb P}}

\newcommand{\QQ }{{\mathbb Q}}
\newcommand{\CC }{{\mathbb C}}
\newcommand{\ZZ }{{\mathbb Z}}

\newcommand{\DD}{\mathsf{D}}

\newcommand{\T}{{\mathsf{T}}}

\newcommand{\lan}{\langle}
\newcommand{\ran}{\rangle}

\def \proj {{\mathbb{P}}}
\newcommand{\com}{{\mathbb{C}}}

\newcommand{\psit}{\theta}
\newcommand{\Psit}{\Theta}

\begin{document}

\title[Crepant resolution
and holomorphic anomaly]{Crepant resolution 
and the holomorphic anomaly equation for $[\CC^3/\ZZ_3]$}

\author{Hyenho Lho}
\address{Department of Mathematics, ETH Z\"urich}
\email {hyenho.lho@math.ethz.ch}
\author{Rahul Pandharipande}
\address{Department of Mathematics, ETH Z\"urich}
\email {rahul@math.ethz.ch}
\date{February 2019, \ \  MSC2010: 14H10, 53D45}

\begin{abstract}
 We study  the orbifold Gromov-Witten theory 
 of the quotient $[\CC^3/\ZZ_3]$ 
 in all genera. Our first result is a proof of the holomorphic anomaly equations in the precise form predicted by $B$-model physics. 
 %The results here are an orbifold analog of our parallel study of local $\PP^2$.
 Our second result is an exact 
 crepant resolution correspondence
 relating the Gromov-Witten theories
 of $[\CC^3/\ZZ_3]$ and local $\PP^2$.
 The proof of the correspondence
 requires an identity proven
 in the Appendix by T. Coates and
 H. Iritani.
 
\end{abstract}

\maketitle

\setcounter{tocdepth}{1} 
\tableofcontents

 \pagestyle{plain}

\setcounter{section}{-1}

\section{Introduction}

\subsection{Overview}
Let $\ZZ_3$ be the cyclic group of order
3 with generator $\omega$.
Let $\ZZ_3$ act on $\CC^3$ by
\begin{equation}
    \label{mmmm}
\omega \, \mapsto\,  \left( \begin{matrix} e^{\frac{2\pi i}{3}} & 0 & 0 \\
0&  e^{\frac{2\pi i}{3}} & 0\\
0 & 0 & e^{\frac{2\pi i}{3}} \end{matrix}\right)\, .
\end{equation}
The central object of our paper is
the orbifold (or Deligne-Mumford stack)
quotient $[\CC^3/\ZZ_3]$. 
The constant holomorphic 3-form of $\CC^3$
descends to $\CC^3/\ZZ_3$ by
the specific choice of the $\ZZ_3$-action \eqref{mmmm}.
We refer
the reader to \cite{Adem,ChenR} for an introduction
to orbifolds
and orbifold cohomology.

Viewed as a singular scheme, the quotient
$\CC^3/\ZZ_3$ admits a crepant resolution
in the category of schemes
by the total space $K\PP^2$ of the canonical bundle of $\PP^2$, 
\begin{equation}\label{ddxx}
K\PP^2 \rightarrow \CC^3/\ZZ_3\, .
\end{equation}
Since $[\CC^3/\ZZ_3]$ is  nonsingular
as an
orbifold, the morphism
$$[\CC^3/\ZZ_3] \rightarrow \CC^3/\ZZ_3\, $$
may be viewed as a crepant resolution in the category
of orbifolds.

Our study of the orbifold Gromov-Witten
theory of $[\CC^3/\ZZ_3]$ in all
genera yields two main results:
\begin{enumerate}
    \item [(i)] We prove
    the holomorphic
    anomaly equations for 
    $[\CC^3/\ZZ_3]$ in the precise
    form predicted by $B$-model physics \cite{ASYZ}.
    \item[(ii)] We prove an
    exact crepant resolution correspondence
    in all genera relating
    the Gromov-Witten theories of $K\PP^2$
    and
    $[\CC^3/\ZZ_3]$.
\end{enumerate}
For (i), our approach follows the path of the
higher genus study in \cite{LP,LP2}.
For (ii), our correspondence is
simple, explicit, and carries no
unevaluated{\footnote{The
general statement of the crepant resolution correspondence
(see \cite[Conjecture 4.1]{CR})
asserts an equivalence up to {\em a choice of analytic continuation}
which is not explicitly specified.}}
analytic continuation.

\subsection{Crepant resolutions}
Following Ruan \cite{Ruan}, 
Bryan-Graber \cite{BT}, 
and, especially for $\CC^3/\ZZ_3$,
Coates-Iritani-Tseng \cite{CIT}, 
the relationship between the Gromov-Witten theories of  
scheme and orbifold crepant
resolutions has been studied for more than
a decade.
In many cases where the
exceptional locus of the resolution
is of dimension 1, a crepant
resolution correspondence is proven
by matching closed
form calculations of the two sides.{\footnote{See, for example,  \cite{BG, BGP, WKC,CCIT,Ross}.}}
However, for the quotient $\CC^3/\ZZ_3$,
the resolution \eqref{ddxx} has
exceptional locus $\PP^2$ of dimension 2, 
and closed forms are not available.

Our correspondence is proven instead
by Cohomological Field Theory (CohFT) techniques \cite{SS,Te}. The crepant
resolution correspondence of \cite{PHH}
for
$$\text{Hilb}(\CC^2,n) \rightarrow
(\CC^2)^n/\Sigma_n$$
is another recent application of
the CohFT perspective. While
the statements of the correspondences
for  $\CC^3/\ZZ_3$
and $(\CC^2)^n/\Sigma_n$
have no unevaluated
analytic continuations,
the proofs both require delicate
identities governing the
changes of variables.  In
the latter case, the result
required is related to the analytic
continuation of the quantum
differential equation 
studied in \cite{OPA}. For $\CC^3/\ZZ_3$,
we require here 
identities proven in the Appendix
by T. Coates and H. Iritani.

\subsection{Orbifold cohomology}
The diagonal $\mathsf{T}=(\com^*)^3$ action on $\CC^3$
descends to the orbifold $[\CC^3/\ZZ_3]$.
The $\mathsf{T}$-equivariant {\em Chen-Ruan orbifold cohomology}
$H^*_{\mathsf{T},\mathsf{orb}}([\CC^3/\ZZ_3])$ has three
canonical elements{\footnote{Cohomology will always
be taken here with $\mathbb{C}$-coefficients.}},
\begin{eqnarray*}
\mathsf{1}=\phi_0 &\in& H^0_{\mathsf{T},\mathsf{orb}}([\CC^3/\ZZ_3])\,, \\
\phi_1 & \in & 
H^2_{\mathsf{T},\mathsf{orb}}([\CC^3/\ZZ_3])\, , \\
\phi_2 & \in & 
H^4_{\mathsf{T},\mathsf{orb}}([\CC^3/\ZZ_3])\, ,
\end{eqnarray*}
which span a basis.
The classes $\phi_1$ and $\phi_2$
correspond (on the
inertial stack) respectively
to the two stabilizers $\omega$ and
$\omega^2\in \ZZ_3$ 
of the fixed point $0\in \CC^3$.

Let $\lambda_0,\lambda_1,\lambda_2$ denote the 3
weights of the representations of $\mathsf{T}$
on the 3 factors of $\CC^3$.
The pairing matrix for
$H^*_{\mathsf{T},\mathsf{orb}}([\CC^3/\ZZ_3])$
in the basis $\phi_0,\phi_1,\phi_2$ is
defined by residues with respect to 
localization by $\mathsf{T}$:
\begin{equation}
\label{IPP}
   \frac{1}{3}
  \left[ {\begin{array}{ccc}
   \frac{1}{\lambda_0\lambda_1\lambda_2}         & 0 & 0 \\
   0         & 0 & 1 \\
   0         & 1 & 0 \\
  \end{array} } \right]\,. 
\end{equation}

For the rest of the paper,
  the specialization
  \begin{equation} \label{specc}
  \lambda_k =e^{\frac{2\pi i k}{3}}
  \end{equation}
  will be fixed. All homogeneous rational functions of degree 0
  in the weights become constants after the specialization \eqref{specc}.

\subsection{Holomorphic anomaly for $[\CC^3/\ZZ_3]$}
\label{vvv234}
The holomorphic anomaly equation has origins in $B$-model physics.
An interpretation of the $B$-model invariants in terms of stable quotients
invariants and a systematic study of the holomorphic anomaly for $K\PP^2$
was given in \cite{LP}. By the crepant resolution philosophy, parallel
holomorphic anomaly equations must hold for $[\CC^3/\ZZ_3]$.
 We state here the precise form of the holomorphic anomaly equations for $[\CC^3/\ZZ_3]$ as predicted by \cite{ASYZ}.
 
 Define the 
 orbifold Gromov-Witten potential of
 $[\CC^3/\ZZ_3]$ for $g\geq 2$ by
 \begin{eqnarray}
 \mathcal{F}_{g}^{[\CC^3/\ZZ_3]} &=& 
 \label{pp99}
 \sum_{d=0}^\infty \frac{\Psit^d}{d!} \int_{[\overline{M}^{\mathsf{orb}}_{g,d}([\CC^3/\ZZ^3],0)]^{\mathsf{T},vir}}\prod_{i=1}^{d}
 {\text{ev}}_i^*(\phi_1)\,, 
 \end{eqnarray}
% \begin{eqnarray}
% \mathcal{F}_{g,n}^{[\CC^3/\ZZ_3]} &=& \lan\lan\, \underbrace{\phi_1,\dots,\phi_1}_{n} %\, \ran \ran_{g,n}^{[\CC^3/\ZZ_3]}
% \label{pp99}
% \\ \nonumber
% & = & \sum_{d=0}^\infty \frac{\Psit^d}{d!} \int_{[\overline{M}^{\mathsf{orb}}_{g,n+d}([\CC^3/\ZZ^3],0)]^{\mathsf{T},vir}}\prod_{i=1}^{n+d}
% {\text{ev}}_i^*(\phi_1)\,, 
% \end{eqnarray}
  where $\phi_1\in H^2_{\mathsf{orb}}([\CC^3/\ZZ_3])$ 
  is the basis element of degree 2.
 % and
 % the pair $(g,n)$ is in the stable range
 % $$2g-2+n>0\, .$$
   The integral on the right side of \eqref{pp99},
 defined by residues with respect to localization
 by $\mathsf{T}$, is a symmetric homogeneous rational function with
 $\mathbb{Q}$-coefficients of degree 0 in the weights $\lambda_0,\lambda_1,\lambda_2$.
 Hence, after the specialization \eqref{specc},
 %{\em independent}{\footnote{The virtual
 %dimension of
% \overline{M}^{\mathsf{orb}}_{g,n+d}([\CC^3/\ZZ_3],0)$ is 0, but we need an argument %for this ...
% I think there is a constant term which depends on
% the weights.}} of $\lambda_0,\lambda_1,\lambda_2$,
$$\mathcal{F}_{g}^{[\CC^3/\ZZ_3]}
\in {\mathbb{Q}}[[\Theta]]\, .$$
%  If $n=0$, we will usually omit the subscript. 
%    While the potentials
% \eqref{pp99} are unaffected by the
% specialization, various auxillary
% $\mathsf{T}$-equivariant functions are afffected.
The definition of the potential in genus 0 and 1 requires insertions
and will be discussed  in Section \ref{111ooo}.

From the $I$-function{\footnote{The formula for the $I$-function is {\em after} the
specialization \eqref{specc}. 
The full $I$-function and the series $I_i$ are
defined in Section \ref{111ooo}.}}
for $[\CC^3/\ZZ^3]$, we obtain 
\begin{equation}\label{Ifmap}
I^{[\CC^3/\ZZ^3]}_1(\psit)=\sum_{n \ge 0}\frac{(-1)^{3n} \psit^{3n+1}}{(3n+1)!}\left(\frac{\Gamma(n+\frac{1}{3})}{\Gamma(\frac{1}{3})}\right)^3\,
 .\end{equation}
Define the mirror map $T(\psit)$ by
\begin{align}\label{MM}
T(\psit)= I^{[\CC^3/\ZZ^3]}_1(\psit)\, .\end{align}

 In order to state the holomorphic anomaly equations, we require the following additional series in $\psit$:
 \begin{eqnarray*}
L(\psit)&=& -\psit(1+\frac{\psit^3}{27})^{-\frac{1}{3}}\, = -\psit+\frac{\psit^4}{81}-\frac{2\psit^7}{6561}+\frac{14 \psit^{10}}{1594323}+\, \ldots\, ,\\
C_1(\psit)&=& \psit \frac{d}{d\psit} I_1^{[\CC^3/\ZZ^3]}=\psit-\frac{\psit^4}{162}+\frac{4\psit^7}{32805}+\, \ldots \, ,\\
A_2(\psit)&=& \frac{1}{L^3}\left(3\frac{\psit\frac{d}{d\psit}C_1}{C_1}
-3 -\frac{L^3}{18}\right)=\frac{\psit^3}{4860}-\frac{41\psit^6}{472392}+\, \ldots\, .
\end{eqnarray*}

The ring $\mathbb{C}[L^{\pm1}]=\mathbb{C}[L,L^{-1}]$ will play a basic role.
%Since 
%$$q = \frac{1}{27}\left(L^{-3}-1\right),$$
%we have $\mathbb{C}[q] \subset \mathbb{C}[L^{\pm1}]$.
Consider the
free polynomial rings in the variables $A_2$ and $C_1^{-1}$ over $\com[L^{\pm1}]$,
\begin{equation}
\label{dsdsds7}
\mathbb{C}[L^{\pm1}][A_2]\, , \ \ \ 
\mathbb{C}[L^{\pm1}][A_2,C_1^{-1}]\,.
\end{equation}
There are canonical maps
\begin{equation}\label{dsdsds22}
\mathbb{C}[L^{\pm1}][A_2]\rightarrow \mathbb{C}[[\psit]]\, , \ \ \ 
\mathbb{C}[L^{\pm1}][A_2,C_1^{-1}]\rightarrow \mathbb{C}((\psit))
\end{equation}
given by assigning the above defined series $A_2(\psit)$ and
$C_1^{-1}(\psit)$ to the variables $A_2$ and $C_1^{-1}$ respectively.
We may therefore consider elements of the rings \eqref{dsdsds7} {\em either} 
as free polynomials in the variables $A_2$ and $C_{1}^{-1}$ {\em or}
as series in $\psit$.

Let $F(\psit)\in \mathbb{C}[[\psit]]$ be a series in $\psit$. When we write
$$F(\psit) \in \mathbb{C}[L^{\pm1}][A_2]\, ,$$
we mean there is a canonical lift ${F}\in \mathbb{C}[L^{\pm1}][A_2]$
for which
$${F} \mapsto F(\psit) \in \mathbb{C}[[\psit]]$$
under the map \eqref{dsdsds22}. 
The symbol $F$ {\em without the
argument $\psit$} is the lift.
The notation 
$$F(\psit) \in \mathbb{C}[L^{\pm1}][A_2,C_1^{-1}]$$
is parallel.

%\textcolor{red}
{Since the holomorphic anomaly equations originate in the B-model, 
we will consider the orbifold Gromov-Witten potential $\mathcal{F}_g^{[\CC^3/\ZZ_3]}$ as a series in $\theta$ via the mirror map \eqref{MM},
$$\Theta = T(\theta)\, .$$
The potential $\mathcal{F}_g^{[\CC^3/\ZZ_3]}$ viewed 
as a series in $\theta$ will be connected to
the stable quotients series of $K\PP^2$. 
%For the rest of the paper, we will always consider $\mathcal{F}_g^{[\CC^3/\ZZ_3]}$ as a series in $\theta$.
%}
 
%\textcolor{red}{In the following Theorems, we recall the mirror map
%\eqref{MM}.}

\begin{Thm} \label{ooo} 
The orbifold Gromov-Witten potentials of
$[\CC^3/\ZZ^3]$ satisfy:
\begin{enumerate}
\item[(i)]
$\mathcal{F}_g^{[\CC^3/\ZZ_3]} (\psit) \in \mathbb{C}[L^{\pm1}][A_2]$
for  $g\geq 2$, \vspace{5pt}
\item[(ii)]
$\mathcal{F}_g^{[\CC^3/\ZZ_3]}$ is of degree at most $3g-3$  with respect to $A_2$,
\item[(iii)]
$\frac{\partial^k \mathcal{F}_g^{[\CC^3/\ZZ_3]}}{\partial T^k}(\psit) \in \mathbb{C}[L^{\pm1}][A_2,C_1^{-1}]$ for  $g\geq 1$ and  $k\geq 1$,
\vspace{5pt}

\item[(iv)]
${\frac{\partial^k \mathcal{F}_g^{[\CC^3/\ZZ_3]}}{\partial T^k}}$ is homogeneous of degree $k$
with respect to  $C_1^{-1}$.

\end{enumerate}
\end{Thm}

\begin{Thm} \label{ttt} The holomorphic anomaly equations
for the orbifold Gromov-Witten invariants of $[\CC^3/\ZZ^3]$
hold for $g\geq 2$:
$$\frac{1}{C_1^2}\frac{\partial \mathcal{F}_g^{[\CC^3/\ZZ_3]}}{\partial{A_2}}
= \frac{1}{2}\sum_{i=1}^{g-1} 
\frac{\partial \mathcal{F}_{g-i}^{[\CC^3/\ZZ_3]}}{\partial{T}}
\frac{\partial \mathcal{F}_i^{[\CC^3/\ZZ_3]}}{\partial{T}}
+
\frac{1}{2}
\frac{\partial^2 \mathcal{F}_{g-1}^{[\CC^3/\ZZ_3]}}{\partial{T}^2}\,
.$$
%For genus $1$,

%$$\frac{\partial \mathcal{F}_1^{\mathsf{SQ}}}{\partial{T}}=-\frac{1}{6 %C_1}L^3 A_2. $$
\end{Thm}

The derivative of the lift $\mathcal{F}_g^{[\CC^3/\ZZ_3]}$ 
with respect to $A_2$ in the holomorphic anomaly
equation of Theorem \ref{ttt} is well-defined since 
$${\mathcal{F}_g^{[\CC^3/\ZZ_3]}}\in \mathbb{C}[L^{\pm1}][A_2]$$
by Theorem \ref{ooo} part (i).
By Theorem \ref{ooo} parts (ii) and (iii),
$$\frac{\partial \mathcal{F}_{g-i}^{[\CC^3/\ZZ_3]}}{\partial T}\,, \
\frac{\partial \mathcal{F}_{i}^{[\CC^3/\ZZ_3]}}{\partial T}\,, \
\frac{\partial^2 \mathcal{F}_{g-1}^{[\CC^3/\ZZ_3]}}{\partial{T}^2}\ 
\in \ \mathbb{C}[L^{\pm1}][A_2,C_1^{-1}]\, $$
are all of degree 2 in $C_1^{-1}$. Hence, 
the holomorphic anomaly equation of Theorem \ref{ttt} may be viewed
as holding 
in $\mathbb{C}[L^{\pm1}][A_2]$
since the factors of $C_1^{-1}$ on the left and right sides cancel.
The
holomorphic anomaly equations here for
$[\CC^3/\ZZ_3]$ are exactly as presented in \cite[(4.27)]{ASYZ} via
$B$-model physics.

Theorem \ref{ttt} determines
$\mathcal{F}_g^{[\CC^3/\ZZ_3]}\in \mathbb{C}[L^{\pm1}][A_2]$ uniquely as a polynomial
in $A_2$ up to a constant term in $\mathbb{C}[L^{\pm1}]$.
In fact, the degree of the constant term can be bounded
(as will be seen in the proof of Theorem \ref{ttt}). So
Theorem \ref{ttt} determines $\mathcal{F}_g^{[\CC^3/\ZZ_3]}$ 
from the lower genus theory together with a finite amount of data.

\subsection{Crepant resolution correspondence}
We start by defining
a the polynomial ring over $\CC[L^{\pm}]$
in a new variable $X$,
$$\mathcal{A}^{[\CC^3/\ZZ_3]}=\CC[L^{\pm 1}][X]\, . $$
By setting
$$X=\frac{\psit \frac{d}{d\psit} C_1}{C_1}
=\frac{1}{3}L^3 A_2 + 1 + \frac{L^3}{54}\, ,
$$
we obtain an isomorphism
$$\mathcal{A}^{[\CC^3/\ZZ_3]}\stackrel{\sim}{=} \CC[L^{\pm 1}][A_2]\, .$$
As a series in $\psit$,
$$X=1-\frac{\psit^3}{54}+\frac{\psit^6}{1620}+\,\ldots\,.$$
Then,
by Theorem \ref{ooo} part (i), we have
$$\mathcal{F}^{[\CC^3/\ZZ_3]}_g\in \mathcal{A}^{[\CC^3/\ZZ_3]}\,.$$

To state the crepant resolution correspondence, we  require results
from our study of $K\PP^2$ in \cite{LP}.
The following series in $q$ were
defined in \cite[Section 0.4]{LP}:
\begin{align*}
    L^{K\PP^2}(q)&=(1+27q)^{-\frac{1}{3}}=1-9q+162q^2+\dots\,,\\
    C_1^{K\PP^2}(q)&=q\frac{d}{dq}I_1^{K\PP^2}=1-6q+90q^2+\dots\,,\\
    X^{K\PP^2}(q)&=\frac{q\frac{d}{dq}C_1^{K\PP^2}}{C_1^{K\PP^2}}=-6q+144q^2+\dots\,.
\end{align*}
Denote the ring generated by $X^{K\PP^2}$ over the base ring $\CC[(L^{K\PP^2})^{\pm 1}]$ by 
$$\mathcal{A}^{K\PP^2}=\CC[(L^{K\PP^2})^{\pm 1}][X^{K\PP^2}]\,.$$
\begin{Thm}{\em (\cite[Theorem 1]{LP})}
 For $g\geq 2$, 
 the stable quotients
 potential  $\mathcal{F}^{K\PP^2}_g$
 satisfies
$$   \mathcal{F}^{K\PP^2}_g \in \mathcal{A}^{K\PP^2}\,.
$$
\end{Thm}

Our crepant resolution correspondence
is based upon a simple ring homomorphism 
\begin{align}\label{RH}
    \mathsf{P}:\mathcal{A}^{K\PP^2}\rightarrow \mathcal{A}^{[\CC^3/\ZZ_3]}
\end{align} defined by
$$
\mathsf{P}(L^{K\PP^2})=-\frac{L}{3}\,,\ \ \
    \mathsf{P}(X^{K\PP^2})=-\frac{X}{3}\,.
$$

\begin{Thm}\label{crc}
 For $g \ge 2$, a crepant resolution
 correspondence holds:
 \begin{align*}
 \mathcal{F}^{[\CC^3/\ZZ_3]}_g =
     \mathsf{P}(\mathcal{F}^{K\PP^2}_g)\,.
 \end{align*}
\end{Thm}

Theorem \ref{crc} is stated on the B-model side since
we use the variables $\theta$ and $q$.
%Actually it is interesting question whether one can define orbifold quasimap space which directly yields $\mathcal{F}^{[\CC^3/\ZZ_3]}_g$ as a series in $\theta$. 
By the mirror maps on both sides, Theorem \ref{crc} is also a direct relationship
between the Gromov-Witten theories.
Knowledge of one side easily determines the other.
A parallel statement for genus 0
and 1 (requiring insertions for
stability) 
is presented in
Section \ref{crc}.
A different approach to the crepant resolution
correspondence for $[\CC^3/\ZZ_3]$ will 
appear in the upcoming paper \cite{CoatIr}.

Theorem \ref{crc} concerns the higher genus potentials
{\em after} specializing 
the equivariant parameters by 
\begin{equation*} 
  \lambda_k =e^{\frac{2\pi i k}{3}}
  \end{equation*}
  on both the $[\CC^3/\ZZ_3]$ and $K\PP^2$ sides.
  In fact, because of compactness{\footnote{{In the
  presence of orbifold markings on the domain curve, the map must factor through $B\ZZ_3$.}}} of the moduli spaces,
  all coefficients of $\mathcal{F}_g^{[\CC^3/\ZZ_3]}$
  of positive degree in $\Psit$ are constant {\em before}
  specialization of the equivariant parameters.
  Similarly, all coefficients of $\mathcal{F}_g^{K\PP^2}$
  of positive degree in $q$ are constant{\footnote{In positive degree, the map must have image in the $0$-section $\PP^2$.}}
   {\em before} specialization.
  So our specialization of the equivariant parameters
  affects {\em only} the leading terms of 
  $\mathcal{F}_g^{[\CC^3/\ZZ_3]}$ and 
  $\mathcal{F}_g^{K\PP^2}$.

\subsection{Acknowledgments} 
We thank R.~Cavalieri, J. Bryan,
T.~Graber, J. Gu\'er\'e,
F. Janda, B. Kim,
A.~Klemm, Y.-P. Lee, M. C.-C. Liu,
M.~Mari\~no, Y. Ruan, and H.-H. Tseng
for many discussions  
about orbifold Gromov-Witten theory.
We are especially grateful to E. Scheidegger
for sharing his calculations in low genus at the
beginning of the project and to
T. Coates and H. Iritani
for contributing the Appendix (which
plays a crucial role in the crepant
resolution result).

R.P. was partially supported by 
% SNF-200021143\-274, 
SNF-200020-162928, SNF-200020\--182181, ERC-2012-AdG-320368-MCSK, 
ERC-2017-AdG-786580-MACI,
SwissMAP, and
the Einstein Stiftung. 
The paper was completed while R.P. was visiting
the {\em Centre for Mathematical Sciences} at
Cambridge University.
H.L. was supported by the grants ERC-2012-AdG-320368-MCSK and
ERC-2017-AdG-786580-MACI.

This project has received funding from the European Research Council (ERC)
under the European Union's Horizon 2020 research and innovation
program (grant agreement No. 786580).

\section{Orbifold Gromov-Witten invariants of $[\CC^3/\ZZ_3]$}
\label{111ooo}
Let $\phi_{a_1}, \ldots, \phi_{a_n}\in H^*_{\mathsf{T},\mathsf{orb}}([\CC^3/\ZZ_3])$.
We define
the Gromov-Witten potential by
\begin{multline} \label{frfr45}
{\mathcal{F}}_{g,n}^{[\CC^3/\ZZ_3]}(\phi_{a_1},\dots,\phi_{a_n}) 
= \\
  \sum_{d=0}^\infty \frac{\Psit^d}{d!} \int_{[\overline{M}^{\mathsf{orb}}_{g,n+d}([\CC^3/\ZZ_3],0)]^{vir}}\prod_{k=1}^n{\text{ev}}_i^*(\phi_{a_k})\prod_{i=n+1}^{n+d}
 {\text{ev}}_i^*(\phi_1)\,.
 \end{multline}
{For the positive coefficients of $\Theta$, the stable map factors through $$B\ZZ_3\subset [\CC^3/\ZZ_3]$$ since there are orbifold markings on the domain curves. For the constant terms, the integrals on the right
side of \eqref{frfr45} are defined via $\mathsf{T}$-equivariant
residues.
If the pair $(g,n+d)$ is not in the stable range,
$$2g-2+n+d >0\, ,$$
the moduli space 
$\overline{M}^{\mathsf{orb}}_{g,n+d}([\CC^3/\ZZ_3],0)$
is empty and the corresponding term in \eqref{frfr45}
vanishes. We will also use the standard double bracket notation
$${\mathcal{F}}_{g,n}^{[\CC^3/\ZZ_3]}(\phi_{a_1},\dots,\phi_{a_n}) 
=
\lan\lan\, \phi_{a_1},\dots,\phi_{a_n} \, \ran \ran_{g,n}^{[\CC^3/\ZZ_3]} \, .$$
For a beautiful introduction to the geometry of stable maps to orbifolds
and 
the Gromov-Witten theory of
$[\CC^3/\ZZ_3]$, we refer the reader to \cite[Section 1]{BoCa}.

The small $J$-function  of $[\CC^3/\ZZ_3]$ 
is defined by
$$
    J^{[\CC^3/\ZZ_3]}(\Psit)= \phi_0+\frac{\Psit \phi_1}{z}+\sum_{i=0}^{2} \big\lan\big\lan\frac{\phi_i}{z(z-\psi)} \big\ran\big\ran^{[\CC^3/\ZZ_3]}_{0,1} \phi^i\, .
$$
Here, $\phi^0, \phi^1, \phi^2$ is the basis of
$H^*_{\mathsf{T},\mathsf{orb}}([\CC^3/\ZZ_3])$
dual $\phi_0,\phi_1,\phi_2$ with respect to the pairing \eqref{IPP}.
After the specialization \eqref{specc}, we have
$$\phi^0=3\,\phi_0\, , \ \ 
\phi^1=3\,\phi_2\, , \ \ \phi^2=3\,\phi_1\, .
$$

The small $I$-function of $[\CC^3/\ZZ_3]$ 
is defined in \cite[Section 6.3]{CCIT} by
\begin{equation} \label{kk77}
    I^{[\CC^3/\ZZ_3]}(\psit)= \sum_{i = 0}
    ^\infty \frac{\psit^i}{z^i i!}\prod_{\substack{0 \le k < \frac{i}{3} \\ [k]=[\frac{i}{3}]}} \left(1-(kz)^3\right)\phi_i\, .
    \end{equation}
The elements $\phi_i$ occur
above for all non-negative integers $i$
via the conventions
$$\phi_0=\phi_3=\phi_6= \ldots= \phi_{3k}=\ldots\, ,\ \ \ \
\phi_1=\phi_4=\phi_7= \ldots= \phi_{3k+1}=\ldots\, ,$$
$$\phi_2=\phi_5=\phi_8= \ldots= \phi_{3k+2}=\ldots\, .$$
There are no positive powers of $z$
on the side of \eqref{kk77}.
Moreover, the coefficient of
$z^{-i}$ always has basis vector $\phi_i$.
Hence, we can
define the functions $I_i(\psit)$ by
$$I^{[\CC^3/\ZZ_3]}(\psit)
    =\sum_{i=0}^\infty\frac{I_i(\psit)}{z^i}\phi_i\, .$$
In particular, $I_1$
is given \eqref{Ifmap}.

%Here we need some explanation for the notations. $\phi^i$ is the dual element of $\phi_i$ with respect to the pairing \eqref{IPP}.
%For $k\ge3$, $\phi_k$ is understood as %$\phi_j$ where $j\in\{0,1,2\}$ satisfy
%$$j=k\,\,\,(\text{mod}\,\, 3)\,.$$

The $I$-function satisfies following Picard-Fuchs equation:
\begin{align*}
\left[  \frac{(z\frac{\psit\partial}{\partial \psit})^3}{27}+1-\psit^{-3}\left(z\frac{\psit\partial}{\partial \psit}\right)
  \left(z\frac{\psit\partial}{\partial \psit}-z\right)\left(z\frac{\psit\partial}{\partial \psit}-2z\right)\right]
  I^{[\CC^3/\ZZ_3]}(\psit)
  =0\, .  
\end{align*}

\begin{Thm} {\em(Coates-Corti-Iritani-Tseng
\cite[Section 6.3]{CCIT})}
After the change of variables $$\Psit(\psit)=I_1(\psit)\, ,$$ 
the following mirror result holds:
  \begin{align*} 
  J^{[\CC^3/\ZZ_3]}(\Psit(\psit)) = I^{[\CC^3/\ZZ_3]}(\psit)\, .
  \end{align*} \label{xx678}
\end{Thm}

\section{Semisimple Frobenius manifolds}

\subsection{Frobenius manifolds}

We briefly review here Givental's
formula for the higher genus
theory associated to
a semisimple Frobenius manifold.
We refer the reader to 
\cite{SS,Book,Picm,PPZ, Te} for more
leisurely treatments.

\begin{Def}
 A Frobenius manifold $(\mathbf{M},\gG,\bullet, A, \mathbf{1})$ satisfies the following conditions:
\begin{enumerate}
    \item[(i)] $\gG$ is Riemmanian metric on  $\mathbf{M}$,
    \item[(ii)] $\bullet$ is commutative and associative product on $T\mathbf{M}$,
    \item[(iii)] A is a symmetric tensor,
     $$A:T\mathbf{M} \otimes T\mathbf{M} \otimes T\mathbf{M} \rightarrow \mathcal{O}_{\mathbf{M}}\, ,   $$
    \item[(iv)] $\gG(X\bullet Y, Z)=A(X,Y,Z)$,
    \item[(v)] $\mathbf{1}$ is a $\gG$-flat unit vector field.
\end{enumerate}
\end{Def}

For us, $\mathbf{M}$ will be a
complex manifold of dimension
$m$. The metric $\gG$
will be symmetric and non-degenerate, but
the positivity condition of a Riemmanian
metric will be dropped (and is not
necessary for the theory).

\subsection{Flat coordinates}

Let $p$ be a point of $\mathbf{M}$. As $\gG$ is flat, we can find flat coordinates $(t^0,t^1,\dots,t^{m-1})$ in a neighborhood of $p$. Let
$$\phi_i=\frac{\partial}{\partial t^i}$$
denote the corresponding flat vector fields. By convention,
$$ \mathbf{1}= \phi_0\, .$$

% Let $g_{ij}=g(\phi_i,\phi_j)$, and let $g^{ij}$ denote the inverse matrix. 
%  By flatness, $g_{ij}$ and $g^{ij}$ are constant matrices.

\subsection{Semisimple points and canonical coordinates}
A point $$p\in \mathbf{M}$$ is  {\em semisimple} if the tangent algebra $(T_p\mathbf{M},\bullet, \mathbf{1})$ is semisimple.
For a semisimple point $p$, we can find {\em canonical coordinates} $$(u^0,u^1,\dots,u^{m-1})$$ in a neighborhood of $p$ for which the
corresponding vector fields
$$e_i=\frac{\partial}{\partial u^i} $$
are orthogonal idempotents: 
$$
e_i \bullet e_j = \delta _{ij} e_i\, $$
and $\gG(e_i,e_j)=0$ for $i\neq j$.

A {\em normalized canonical basis}
$\{\widetilde{e}_i\}$ is constructed by 
$$ \widetilde{e}_i = \gG(e_i,e_i)^{-\frac{1}{2}} e_i\, . $$
The normalized coordinates require choices
of square roots (but the final formulas
are independent of these choices).

Let $\mathbf{\Psi}$ be the transition matrix from the basis $\{\phi_i\}$ to the basis $\{\widetilde{e}_\alpha\}$.
By the orthonormality of $\widetilde{e}_\alpha$, the elements of $\mathbf{\Psi}$ are
$$\mathbf{\Psi}_{\alpha i}=g(\widetilde{e}_\alpha,\phi_i)\, .$$

\subsection{Fundamental solutions and the $\mathbf{R}$-matrix}

We define $$\mathbf{R}(z)=\sum_{k=0}^{\infty}\mathbf{R}_k z^k$$ by following flatness equation:
\begin{align}\label{FSM}
    zd\mathbf{\Psi}^{-1}\mathbf{R}+z\mathbf{\Psi}^{-1}d\mathbf{R}+\mathbf{\Psi}^{-1}\mathbf{R}d\mathbf{U}-\mathbf{\Psi}^{-1}d\mathbf{U}\mathbf{R},
\end{align}
where $\mathbf{U}$ is the diagonal
matrix with coefficients
$$\mathbf{U}=\text{Diag}(u^0,u^1,\dots,u^{m-1})\, .$$

The $\mathbf{R}$-matrix $\mathbf{R}(z)$ is uniquely determined by \eqref{FSM} and the {\em symplectic condition},
\begin{equation}
    \label{symmm}
\mathbf{R}(z)\cdot \mathbf{R}^t(-z) = \text{Id}\, ,
\end{equation}
up to  right multiplication by a constant matrix
$$ \text{exp}\left(\sum_{k \ge 1} \mathbf{a}_{2k-1}z^{2k-1}\right)\, ,$$
where the matrices 
$\mathbf{a}_{2k-1}$ are 
diagonal
with constant coefficients
$$\mathbf{a}_{2k-1}=\text{Diag}[a^0_{0,2k-1},a^1_{1,2k-1},
\dots,a^{m-1}_{m-1,2k-1}]\, .$$

The $\mathbf{R}$-matrix determines an endomorphism
$$\mathbf{R}(z) \in \text{End}(T_p\mathbf{M})[[z]]$$
defined in the  basis $\{\widetilde{e}_i\}$.
Given a vector $v\in T_p\mathbf{M}$,
$$\mathbf{R}(z)v \in T_p\mathbf{M}[[z]]\, .$$
\subsection{Higher genus potentials}

\label{hgp}

\subsubsection{Topological field theory}
Let $p\in \mathbf{M}$ be a semisimple point, and 
let $\Omega_{g,n}$ be the {\em Topological Field Theory} on $T_p\mathbf{M}$ defined by 
\begin{align*}
    \gG(u\bullet v,w)=\Omega_{0,3}(u,v,w).
\end{align*}
The CohFT axioms easily yield:
\begin{align*}
    \Omega_{g,n}(\widetilde{e}_{i_1},\widetilde{e}_{i_2},\dots,\widetilde{e}_{i_n})=\left\{ \begin{array}{cl} \sum_{j=0}^{m-1}\gG(e_{j},e_{j})^{1-g}  & \text{if } n=0\,,  \\ \\
    \gG(e_{i_1},e_{i_1})^{-\frac{2g-2+n}{2}} & \text{if } i_1=i_2=\dots=i_n\, ,\\ \\
    0                                & \text{otherwise} \, . \end{array}\right.
\end{align*}
In Section \ref{hgp}, for the higher genus potential,
we will use the basis $\{\widetilde{e}_i\}$
of $T_p\mathbf{M}$ in all formulas.

\subsubsection{Potentials} \label{ffxx2}
%By symplectic condition of $\mathbf{R}$, we can check that the expression
%
%\begin{align*}
%    \left[\frac{(\mathbf{g}^{-1})^t-(\mathbf{\Psi}^{-1}\mathbf{R})^t(\mathbf{g}^{-1})^t(\mathbf{\Psi}^{-1}\mathbf{R})}{z+w}\right]
%\end{align*}
%is a well-defined matrix-valued power series in $z$ and %$w$.

Let $\mathsf{G}_{g,n}$ be the finite set of stable graphs of genus $g$ with $n$ legs. 
Givental's higher genus (cycle valued)
potential functions 
at $p\in \mathbf{M}$ are 
defined by the following formula 
\begin{align*}
    \mathcal{F}_{g,n}(v_1,v_2,\dots,v_n)=\sum_{\Gamma \in \mathsf{G}_{g,n}}\frac{1}{\text{Aut}(\Gamma)} \text{Cont}_\Gamma\, 
\end{align*}
for $v_i \in T_p\mathbf{M}$.
The contributions
$\text{Cont}_\Gamma$ are determined by
\begin{align*}
    \text{Cont}_\Gamma = \xi_{\Gamma*}\left(\prod_{v\in \text{Vert}(\Gamma)}\sum_{k=0}^\infty \frac{1}{k!}\pi_* \Omega_{g(v),n(v)+k}\right)
\end{align*}
where $\xi_\Gamma$ is the standard  map of the stratum indexed
by $\Gamma$,
$$\xi_\Gamma: \overline{M}_\Gamma \rightarrow \overline{M}_{g,n}\,,$$
$\pi$ is the forgetful map at the vertex dropping
the last $k$ markings,
$$\pi: \overline{M}_{g(v),n(v)+k}\rightarrow\overline{M}_{g(v),n(v)}\, ,$$
and the insertions in the arguments of $\prod_v\Omega_{g(v),n(v)+k}$ are specified
by the following rules:
\begin{itemize}
 \item For the insertion corresponding to the $i^{th}$
 original marking, place $\mathbf{R}^{-1}(\psi_i)v_i$.
 \item For each pair of insertions corresponding to an edge, place the bivector
 $$\sum_{ij}\left[\frac{\mathbf{g}^{-1}-\mathbf{R}^{-1}(\psi)\mathbf{g}^{-1}(\mathbf{R}^{-1}(\psi'))^t}
 {\psi+\psi'}\right]_{ij}\widetilde{e}_i\otimes\widetilde{e}_j\in V^{\otimes 2}[[\psi,\psi']]\, , $$
well-defined by symplectic property $\mathbf{R}$.  

Here, $\mathbf{g}$ and $\mathbf{g}^{-1}$ denote
the matrices obtained from the metric in
the {\em normalized canonical basis}. In fact, both are the
identity matrix.

 \item For each additional insertion at a 
 vertex, place
 $$T(\psi)=\psi\left(\text{Id}-\mathbf{R}^{-1}(\psi)\right) \phi_0\, .$$
\end{itemize}

\subsection{Givental-Teleman classification}
Let $\Lambda$ be a semisimple CohFT with 
unit and state space
$(V,\gG,\mathsf{1})$. 
The genus 0 part of $\Lambda$ determines a
Frobenius manifold
structure on the complex vector space $V$
for which
$0\in V$ is a semisimple point.

The Givental-Teleman classification \cite{SS,Te}
 states: {\em there exists a unique 
$\mathbf{R}$-matrix for the
Frobenius manifold $V$ for which Givental's potential
(as defined in Section \ref{ffxx2}) equals the
CohFT evaluation
\begin{align*}
    {\Lambda}_{g,n}(v_1,v_2,\dots,v_n)\in H^*(\overline{M}_{g,n}) 
\end{align*}
for all $g$ and $n$ in the stable range.}

\section{Genus 0 theory for $[\CC^3/\ZZ_3]$}

\subsection{Summary} We review the genus 0 orbifold Gromov-Witten theory of $[\CC^3/\ZZ_3]$. We follow the notations and conventions of \cite{Book}.
The main difficult result that we will use in the genus 0 theory is
the mirror transformation of Theorem \ref{xx678}
proven by Coates-Corti-Iritani-Tseng \cite{CCIT}.
Similar computations appeared in \cite{GR} for the study of genus one FJRW invariants associated to the quintic threefold.

%Define $\Delta$ to be diagonal
%matrix with coefficients
%$$\Delta=
%\text{Diag}\left[ g(e_1,e_1)^{-1},
%g(e_2,e_2)^{-1}, \ldots, g(e_m,e_m)^{-1}
%\right]
%\,.$$
%The transition matrix from $\{\phi_i\}$ to the basis %$\{{e}_\alpha\}$ is
%then $$\sqrt{\Delta} \mathbf{\Psi}\, .$$

\subsection{Frobenius structure}
The orbifold Gromov-Witten theory determines
an Frobenius manifold structure{\footnote{The Frobenius manifold here is over the
ring $\mathbb{C}[[\Theta]]$.}} on $H^*_{\mathsf{T},\mathsf{orb}}([\CC^3/\ZZ_3])$ 
viewed with flat basis $\phi_0,\phi_1,\phi_2$
and with specialization
\eqref{specc}.
The inner product and the quantum product are as
follows.

\begin{itemize}
 \item{\em Inner product.}
 In the flat basis 
 \begin{equation}\label{ffll}
     \{\phi_0,\phi_1,\phi_2\}\,,
     \end{equation}
     the inner product $\mathbf{g}$, given by
 \begin{align}\label{IP}
   \mathbf{g}=\frac{1}{3}
  \left[ {\begin{array}{ccc}
   1         & 0 & 0 \\
   0         & 0 & 1 \\
   0         & 1 & 0 \\
  \end{array} } \right]\, ,
\end{align}
has    already appeared in 
 \eqref{IPP}. 
\vspace{5pt}

\item{\em Potential.} The full genus 0 Gromov-Witten potential is
 a function of the coordinates $\{t_0, t_1,t_2\}$ in the flat basis
\eqref{ffll} and of the additional variable $\Psit$,
\begin{eqnarray*}
 \mathcal{F}_{0}^{[\CC^3/\ZZ_3]}(t,\Psit) 
 & = & \sum_{n=0}^\infty \sum_{d=0}^\infty 
  \int_{[\overline{M}^{\mathsf{orb}}_{g,n+d}([\CC^3/\ZZ^3],0)]^{\mathsf{T},vir}}
\frac{1}{n!d!} \prod_{i=1}^n {\text{ev}}_i^*(\gamma)
 \prod_{i=n+1}^{n+d}
 {\text{ev}}_i^*(\Psit \phi_1)\,, 
 \end{eqnarray*}
where $\gamma= \sum_{i=0}^2 t_i \phi_i$. The potential satisfies
$$\frac{\partial}{\partial t_1}
\mathcal{F}_{0}^{[\CC^3/\ZZ_3]} = 
\frac{\partial}{\partial \Psit} \mathcal{F}_{0}^{[\CC^3/\ZZ_3]}\, .$$

 \item{\em Quantum product.} The 6 products at 
 $0\in H^*_{\T,\mathsf{orb}}([\CC^3/\ZZ_3])$
 are
 \begin{align*}
     &\phi_0 \bullet \phi_0 =\phi_0\, ,\\
     &\phi_0 \bullet \phi_1=\phi_1\, ,\\
     &\phi_0 \bullet \phi_2=\phi_2\, ,\\
     &\phi_1 \bullet \phi_1= -\frac{L^3}{C_1^3}\phi_2\, ,\\
     &\phi_1 \bullet \phi_2= \phi_0\, ,\\
     &\phi_2 \bullet \phi_2= -\frac{C_1^3}{L^3}\phi_1\, .
 \end{align*}
\end{itemize}

Since the quantum product is at 
 $0\in H^*_{\T,\mathsf{orb}}([\CC^3/\ZZ_3])$, only the variable
 $\Psit$ appears in the functions on the right side of
 the above formulas. In fact, both $L$ and $C$ are defined
 in Section \ref{vvv234} in terms of $\psit$, so the $\Theta$
 dependence appears only after inverting the mirror map $\Psit(\psit)$.
We will give a proof of the quantum product in the 
following subsection.

\subsection{Calculation of the quantum product}
To compute the quantum product of $[\CC^3/\ZZ_3]$, we require
the 3-point functions in genus 0.

\begin{Lemma}\label{g0C3} The nonvanishing $3$-point function in genus 0 are:
\begin{align*}
    &\lan\lan \phi_0,\phi_0,\phi_0\ran\ran_{0,3}^{[\CC^3/\ZZ_3]}=\frac{1}{3}\, ,   
    &\lan\lan\phi_0,\phi_1,\phi_2\ran\ran_{0,3}^{[\CC^3/\ZZ_3]}=\frac{1}{3}\, ,\\
    & \lan\lan \phi_1,\phi_1,\phi_1\ran\ran_{0,3}^{[\CC^3/\ZZ_3]}=-\frac{1}{3}\frac{L^3}{C_1^3}\, ,   & \lan\lan \phi_2,\phi_2,\phi_2\ran\ran_{0,3}^{[\CC^3/\ZZ_3]}=-\frac{1}{3}\frac{C_1^3}{L^3}\, .
\end{align*}
For other choices of insertions, the $3$-point functions
in genus 0 vanish.
\end{Lemma}

\begin{proof}
 By \cite{CCIT}, the $I$-function $I^{[\CC^3/\ZZ_3]}$ lies on the Lagrangian{\footnote{
  See \cite{CG,SF} for the definition of the Lagrangian cone. }} cone
  $\mathcal{L}^{[\CC^3/\ZZ_3]}$ encoding the genus $0$ Gromov-Witten theory of $[\CC^3/\ZZ^3]$. By standard properties of the Lagrangian cone, 
  we obtain the following
  results, see for example \cite{KL,RR}:
  \begin{eqnarray}\label{BF}
 \nonumber\mathds{S}^{[\CC^3/\ZZ_3]}(\Psit(\psit),z)(\phi_0)&=&I^{[\CC^3/\ZZ_3]}\,,\\
 \mathds{S}^{[\CC^3/\ZZ_3]}(\Psit(\psit),z)(\phi_1)&=&\frac{z\mathsf{D}\mathds{S}^{[\CC^3/\ZZ_3]}(\phi_0)}{C_1}\, ,\\
 \nonumber\mathds{S}^{[\CC^3/\ZZ_3]}(\Psit(\psit),z)(\phi_2)&=&\frac{z\mathsf{D}\mathds{S}^{[\CC^3/\ZZ_3]}(\phi_1)}{C_2}\, .
 \end{eqnarray}
 Here, the $\mathds{S}$-operator for $[\CC^3/\ZZ_3]$ is defined as usual by
 \begin{align*}
     \mathds{S}^{[\CC^3/\ZZ_3]}(\Psit,z)(\gamma)=\sum_i \phi^i\lan\lan\frac{\phi_i}{z-\psi},\gamma\ran\ran_{0,2}^{[\CC^3/\ZZ_3]}\, , \ \ \ \text{for}\, \gamma\in H^*_{\mathsf{T},\mathsf{orb}}([\CC^3/\ZZ_3])\, ,
 \end{align*}
 The differential operator
$\mathsf{D}=\psit\frac{d}{d\psit}$ acts on $\mathds{S}$ via
variable change $\Psit(\psit)$. The functions
\begin{equation*}
 C_0=1\,,\ \ \
 C_1=\mathsf{D}I_1\,,\ \ \
 C_2=\mathsf{D}\left(\frac{\mathsf{D}I_2}{C_1}\right)\,.
\end{equation*}
appear on the right side
of \eqref{BF}.

Using the methods of \cite[Theorem 2]{ZaZi}, we obtain
$$C_1^2 C_2=-L^3\,.$$
Observe that the $I$-function has following expansion,
\begin{align*}
    I^{[\CC^3/\ZZ_3]}=\phi_0+\frac{I_1 \phi_1}{z}+\frac{I_2 \phi_2}{z^2}+ \mathsf{O}(\frac{1}{z^3})\, .
\end{align*}
 Then, equation \eqref{BF} immediately yields
 \begin{align*}
     &\lan\lan\phi_0,\phi_1,\phi_2\ran\ran^{[\CC^3/\ZZ_3]}_{0,3}=\frac{1}{3}\,,\\
     &\lan\lan\phi_1,\phi_1,\phi_1\ran\ran^{[\CC^3/\ZZ_3]}_{0,3}=\frac{1}{3}\frac{C_2}{C_1}=-\frac{1}{3}\frac{L^3}{C_1^3}\,.
 \end{align*}
By definition of the Frobenius structure,
 \begin{align*}
     \mathbf{g}(X \bullet Y,Z)=\lan\lan X,Y,Z\ran\ran_{0,3}^{[\CC^3/\ZZ_3]}\,, \ \ \ \text{for}\,\, X,Y,Z\in H^*_{\T,\text{orb}}([\CC^3/\ZZ_3])\, .
 \end{align*}
 The remaining two
 evaluations follow from
 the associativity of the quantum product.
\end{proof}

\subsection{Canonical coordinates}
After normalizing the basis $\{\phi_0,\phi_1,\phi_2\}$ by
\begin{equation}\label{nnn34}
    \widetilde{\phi}_0=\phi_0\, ,\ \ \
    \widetilde{\phi}_1=-\frac{C_1}{L}\phi_1\, ,\ \ \
    \widetilde{\phi}_2=-\frac{L}{C_1}\phi_2\, ,
\end{equation}
we obtain the relation
\begin{align*}
    \widetilde{\phi}_i \bullet \widetilde{\phi}_j = \widetilde{\phi}_{i+j}\, .
\end{align*}
The quantum product  at 
 $0\in H^*_{\T,\mathsf{orb}}([\CC^3/\ZZ_3])$
can
then be checked to be semisimple with an idempotent basis,  
\begin{align*}
    e_{\alpha} \bullet e_{\beta}= \delta_{\alpha \beta}e_{\alpha}\, ,
\end{align*}
given by the formula
\begin{align}\label{canbasis}
    e_\alpha=\frac{1}{3}\sum_{i=0}^3 \zeta^{-\alpha i}\widetilde{\phi}_i\,  \ \ \ \text{for}\ \ \ \alpha=0,1,2,
\end{align}
where $\zeta = e^{\frac{2 \pi i}{3}}$ is a third root of unity.
The normalized idempotents are 
\begin{equation}\label{ll99}
    \widetilde{e}_\alpha= \frac{e_\alpha}{\sqrt{\gG(e_\alpha,e_\alpha)}}=3 e_\alpha\, .
\end{equation}
Equations \eqref{nnn34}-\eqref{ll99} take place
at the point
 $0\in H^*_{\T,\mathsf{orb}}([\CC^3/\ZZ_3])$ of the
 Frobenius manifold and hence only depend upon the variable $\Psit$.

Let $\{u^\alpha\}$ be the canonical coordinates associated to
the above idempotent basis with constants fixed by
\begin{equation}\label{kk55}
u^\alpha(t_i=0, \Psit=0)=0\, .
\end{equation}
Since $e_\alpha = \frac{\partial}{\partial u^\alpha}$, we have
\begin{equation} \label{ssee}
\sum_{\alpha=1}^3 e_\alpha \frac{d u^\alpha}{dt_1} = \phi_1 \, . 
\end{equation}
The standard convention for equations such as \eqref{ssee} is
that the derivative $\frac{\partial}{\partial t_1}$ on the left
side is taken {\em before} all the $t_i$ are set to 0.

\begin{Lemma}\label{CAF}
 We have 
 \begin{align*}
     \frac{d u^\alpha}{dt} = \zeta^\alpha \left(-\frac{L}{C_1}\right)
     \end{align*}
\end{Lemma}

\begin{proof}
The result  is a consequence of equations
\eqref{canbasis} and \eqref{ssee}.
\end{proof}

Before restriction to the point $0\in H^*_{\T,\mathsf{orb}}([\CC^3/\ZZ_3])$, the 
genus 0 potential
$\mathcal{F}_{0}^{[\CC^3/\ZZ_3]}$, the components of the
idempotents in flat coordinates, and the canonical coordinates $\{ u^\alpha\}$
all are functions of the variables $t_0,t_1,t_2,\Psit$
which are annihilated by the operator{\footnote{The proof is
elementary starting with the annihilation of the potential
$\mathcal{F}_{0}^{[\CC^3/\ZZ_3]}$.
The components of
$du^\alpha$ in 
the basis $\{dt_i\}$
are  eigenvalues
of matrices
with coefficients all
annihilated by \eqref{ww34},
so also annihilated by \eqref{ww34}.
Hence,
\begin{equation*}
d\left(\frac{\partial u^\alpha}{\partial t_1} - \frac{\partial u^\alpha}{\partial \Theta}\right)=0\, 
\end{equation*}
and hence 
$\frac{\partial u^\alpha}{\partial t_1} - \frac{\partial u^\alpha}{\partial \Theta}$ must be
a function $f^\alpha(\Psit)$ only of
$\Psit$.
Then, we can find {\em unique
canonical coordinates}
(by shifting by the integral $\int f^\alpha(\Psit) d\Psit$
which satisfy \eqref{kk55}
and are annihilated
by \eqref{ww34}.}}
\begin{equation}\label{ww34}
\frac{\partial}{\partial t_1} - \frac{\partial}{\partial \Theta}\, .
\end{equation}
By the argument of \cite[Section 3]{PHH}, the ${\mathbf{R}}$-matrix
of the associated CohFT is also annihilated by \eqref{ww34}.

Since the operator \eqref{ww34} annihilates $u^\alpha$, we can 
rewrite Lemma \ref{CAF} at the point 
$0\in H^*_{\T,\mathsf{orb}}([\CC^3/\ZZ_3])$ using
$$\frac{d u^\alpha}{dt}  = \frac{du^\alpha}{d\Psit} =
\frac{du^\alpha}{d\psit}\frac{d\psit}{d\Psit}\, .$$
We then obtain the equation
$$\frac{du^\alpha}{d\psit}=
\zeta^\alpha  (-L) \frac{1}{\psit}\, .$$

\subsection{Transition matrix}

The transition matrix $\mathbf{\Psi}$ from flat coordinate to normalized canonical basis is given by
$$\mathbf{\Psi}_{\alpha i}=g(\widetilde{e}_\alpha,\phi_i). $$
At the point
$0\in H^*_{\T,\mathsf{orb}}([\CC^3/\ZZ_3])$, we
can calculate using
\eqref{canbasis}:
\[
   \mathbf{\Psi}=\frac{1}{3}
  \left[ {\begin{array}{ccc}
    1 & -\frac{L}{C_1}        & -\frac{C_1}{L} \\
    1 & -\zeta \frac{L}{C_1}  & -\zeta^2\frac{C_1}{L} \\
    1 & -\zeta^2\frac{L}{C_1} & -\zeta\frac{C_1}{L} \\
  \end{array} } \right].
\]
Viewed a functions of
$t_0,t_1,t_2,\Psit$,
the coefficients of 
$\Psit$ are annihilated
by \eqref{ww34}.

\subsection{Fundamental solution matrix}

Consider the coefficient of $z^k$ in the flatness equation \eqref{FSM}. We
will 
study the solutions 
along the line $\{t_0=0,t_2=0\}$, 
so we consider only
the flatness
equation with respect to the $t_1$ directional derivative in $\eqref{FSM}$. 
Using the
annihilation of 
all functions
by \eqref{ww34} and
the change of variable
relation
$$\frac{\partial}{\partial \Psit}=\frac{\psit}{C_1}\frac{ \partial}{\partial \psit}\,,$$
we obtain,
\begin{align}\label{E1}
    \mathbf{\Psi} \left(\frac{\psit \partial}{\partial \psit} \mathbf{\Psi}^{-1}\right) \mathbf{R}_{k-1}+\frac{\psit \partial}{\partial \psit}\mathbf{R}_{k-1}+\mathbf{R}_k \frac{\psit \partial}{\partial \psit}\mathbf{U}-\left(\frac{\psit \partial}{\partial \psit}\mathbf{U}\right)\mathbf{R}_k=0\, ,
\end{align}
or equivalently, in most useful form,
\begin{align*}
%\label{FS}
     \frac{\psit \partial}{\partial \psit} \left( \mathbf{\Psi}^{-1} \mathbf{R}_{k-1} \right) +\left(\mathbf{\Psi}^{-1}\mathbf{R}_k\right) \frac{\psit \partial}{\partial \psit}\mathbf{U}-\mathbf{\Psi}^{-1}\left(\frac{\psit \partial}{\partial \psit}\mathbf{U}\right)\mathbf{\Psi} \left(\mathbf{\Psi}^{-1}\mathbf{R}_k\right)=0\, .
\end{align*}
We then restrict
to the point
$0\in H^*_{\T,\mathsf{orb}}([\CC^3/\ZZ_3])$, so
\eqref{E1} and the 
second form, become equations
purely of the variable $\psit$.

From Lemma \ref{CAF}, we obtain

\[
   \frac{\psit \partial}{\partial \psit}\mathbf{U}=
  \left[ {\begin{array}{ccc}
     -L & 0         & 0 \\
     0  &  -\zeta L  & 0 \\
     0  & 0          & -\zeta^2 L \\
  \end{array} } \right].
\]
We also have
\[
   \mathbf{\Psi}^{-1}=
  \left[ {\begin{array}{ccc}
     1             &  1                    & 1 \\
    -\frac{C_1}{L} & -\zeta^2\frac{C_1}{L} & -\zeta \frac{C_1}{L} \\
    -\frac{L}{C_1} & -\zeta \frac{L}{C_1}  & -\zeta^2\frac{L}{C_1} \\
  \end{array} } \right].
\]

Let $P^k_{ij}$ denote the $(i,j)$ coefficient of the matrix $\mathbf{\Psi}^{-1} \mathbf{R}_k$ restricted
to 
$0\in H^*_{\T,\mathsf{orb}}([\CC^3/\ZZ_3])$.
From the second form of \eqref{E1},
we obtain the following equations for $ j=0,1,2$ :
\begin{align}\label{R}
    \nonumber&\frac{\psit\partial}{\partial\psit}P^{k-1}_{0j}=C_1P^k_{2j}+LP^k_{0j}\zeta^j,\\
    &\frac{\psit\partial}{\partial\psit}P^{k-1}_{1j}=C_1P^k_{0j}+LP^k_{1j}\zeta^j,\\
    \nonumber&\frac{\psit\partial}{\partial\psit}P^{k-1}_{2j}=-\frac{L^3}{C^2_1}P^k_{1j}+LP^k_{2j}\zeta^j\,.
\end{align}

\subsection{Generators and relations}
As before, let
$\DD=\psit \frac{\partial}{\partial\psit}$.
\begin{Lemma}\label{Relation}
 We have the following relation between $L$ and $X=\frac{\DD C_1}{C_1}$:
  \begin{align*}
     &\mathsf{D}L=L\left(\frac{L^3}{27}+1\right)\,,\\
     &X^2-3\frac{\DD L}{L}X+2\frac{\DD L}{L}+\DD X=0\,.
 \end{align*}
\end{Lemma}
\begin{proof}
 The first relation follows from the definition of $L$. The second relation follows from case $k=2$ of \eqref{R}.
\end{proof}

By above result, we view the differential ring
$$\CC[L^{\pm1}][X,\DD X,\DD \DD X,\dots]$$
as simply  the polynomial ring $\CC[L^{\pm1}][X]$.

The following normalizations will be
convenient for us:
\begin{align}\label{RT}
    \nonumber\widetilde{P}^k_{0j} & = P^k_{0j} \zeta^{kj}\\
    \widetilde{P}^k_{1j} & = -\frac{L}{C_1 \zeta^{2j}}P^k_{1j} \zeta^{kj}\\
    \nonumber\widetilde{P}^k_{2j} & = -\frac{C_1}{L \zeta^{j}} P^k_{2j}\zeta^{kj}\, , \ \ \ k\ge0,\ \ j=0,1,2.
\end{align}
From \eqref{R}, we can calculate $\widetilde{P}^k_{ij}$ explicitly with initial conditions
\begin{align}\label{RTI}
    \tilde{P}^k_{ij}|_{\psit=0}=0\, , \ \ \ k \ge 1.
\end{align} For example, for $j=0,1,2$, we have
\begin{align*}
    \widetilde{P}^0_{0j} &= 1,\\
    \widetilde{P}^1_{0j} &=\frac{L^2}{162},\\
    \widetilde{P}^2_{0j} &=\frac{L}{81}+\frac{25}{52488}L^4,\\
    \widetilde{P}^3_{0j} &=\frac{7}{4374}L^3+\frac{1225}{25509168}L^6\, .
\end{align*}

 Using \eqref{R} and Lemma \ref{Relation}, we obtain the Lemma
 \ref{aass} below.
 Lemma \ref{RR} follows 
 from 
 an argument parallel
to \cite[Section 1]{ZaZi}.

\begin{Lemma} \label{aass}
For $j=0,1,2$, we have:
\begin{align}\label{RC3}
 \nonumber \widetilde{P}^{k+1}_{2j} & =\widetilde{P}^{k+1}_{0j}-\frac{\DD \widetilde{P}^k_{0j}}{L},\\
 \widetilde{P}^{k+1}_{1j} & =\widetilde{P}^{k+1}_{2j}-\frac{\DD \widetilde{P}^k_{2j}}{L}-\left(\frac{\DD L}{L^2}-\frac{X}{L}\right)\widetilde{P}^k_{2j},\\
 \nonumber\widetilde{P}^{k+1}_{0j} & =\widetilde{P}^{k+1}_{1j}-\frac{\DD \widetilde{P}^k_{1j}}{L}+\left(\frac{\DD L}{L^2}-\frac{X}{L}\right)\widetilde{P}^k_{1j}\, .
 %\ \ \ i=0,1,2.
\end{align}
\end{Lemma}

\begin{Lemma}\label{RR}
We have
    $\widetilde{P}^k_{0j}\in \CC[L^{\pm1}]$.

\end{Lemma}

The following result is
a direct consequence of Lemmas \ref{aass} and
\ref{RR}.

\begin{Lemma}\label{RR2}
 For all $k \ge 0$ and $j=0,1,2$, we have
 \begin{align*}
     &\widetilde{P}^k_{2j} \in \CC[L^{\pm1}],\\
     &\widetilde{P}^k_{1j} = \widetilde{Q}^k_{1j}+\frac{\widetilde{P}^{k-1}_{2j}}{L}X,
 \end{align*}
with $\widetilde{Q}^k_{1j} \in \CC[L^{\pm1}]$.
\end{Lemma}

\section{The holomorphic anomaly equations}\label{haes}

\subsection{$\mathbf{R}$-matrix}

Let $\widetilde{\mathsf{R}}^{[\CC^3/\ZZ^3]}$ be the matrix whose $z^k$ coefficient is
the solution $\mathbf{R}_{k}$ of \eqref{E1} with initial conditions{\footnote{The symplectic condition \eqref{symmm} is not imposed on (and not satisfied by) $\widetilde{\mathsf{R}}^{[\CC^3/\ZZ^3]}$.}}
\begin{align}\label{Ic}
    \left[\widetilde{\mathsf{R}}^{[\CC^3/\ZZ^3]}(z)\right]\Big|_{\psit=0}=\text{Id}\, .
\end{align}
Define a new diagonal matrix
$$\mathsf{B}(z)= \text{Diag}\, \big[B^{[\CC^3/\ZZ_3]}_0(z),B^{[\CC^3/\ZZ_3]}_1(z),B^{[\CC^3/\ZZ_3]}_2(z)\big]$$
where for i=0,1,2,
$$B^{[\CC^3/\ZZ_3]}_i(z)=\text{Exp}\left(3\sum_{k=1}^{\infty}(-1)^{k+1}\frac{B_{3k+1}\left(\frac{\mathsf{Inv}(i)}{3}\right)}{3k+1}\frac{z^{3k}}{3k}\right)\,.$$
Here, the involution $\mathsf{Inv} : \{0,1,2\}\rightarrow\{0,1,2\}$ is defined by
$$\mathsf{Inv}(0)=0\,,\,\,\,\mathsf{Inv}(1)=2\,,\,\,\,\mathsf{Inv}(2)=1\,.$$
The Bernoulli polynomials $B_m(x)$ are defined by
$$\frac{t e^{tx}}{e^t-1}=\sum_{m\ge 0}\frac{B_m(x)t^m}{m!}\,.$$
For example,
\begin{align*}
    B_0(x)=1\,,\,\,B_1(x)=x-\frac{1}{2}\,,\,\,B_2(x)=x^2-x+\frac{1}{6}\,.
\end{align*}
Especially, $B_k(0)$ is the Bernoulli numbers.

Via the orbifold quantum Riemann-Roch theorem in \cite[Section 4.2]{Ts}, we obtain
the following result.

\begin{Prop}\label{ORR}
 The true $\mathbf{R}$-matrix $\mathsf{R}^{[\CC^3/\ZZ_3]}$ for the Gromov-Witten theory of $[\CC^3/\ZZ_3]$ has the following form after restriction $\psit=0$:
 \begin{align*}
     \left[ \mathsf{R}^{[\CC^3/\ZZ_3]}(z) \right]_{ij}\Big|_{\psit=0}=(\mathbf{\Psi}|_{\psit=0})\cdot \mathsf{B}(z)\cdot(\mathbf{\Psi}|_{\psit=0})^{-1}\,.
 \end{align*}
\end{Prop}

\begin{Cor}
 The true $\mathbf{R}$-matrix $\mathsf{R}^{[\CC^3/\ZZ_3]}$ 
 for the Gromov-Witten theory of $[\CC^3/\ZZ_3]$ 
 in the normalized canonical basis is given by
 \begin{align*}
     \mathsf{R}^{[\CC^3/\ZZ_3]}(z)=(\mathbf{\Psi}|_{\psit=0})\cdot \mathsf{B}(z)\cdot(\mathbf{\Psi}|_{\psit=0})^{-1}\cdot\left[\widetilde{\mathsf{R}}^{[\CC^3/\ZZ^3]}(z)\right]\, .
 \end{align*}
\end{Cor}

\begin{proof}
 The coefficients of the matrices $\mathsf{R}^{[\CC^3/\ZZ_3]}(z)$ and $\widetilde{\mathsf{R}}^{[\CC^3/\ZZ^3]}(z)$ satisfy  the same system of differential equations \eqref{R}. Therefore, the solutions differ by a constant (with respect to $\psit$) matrix which can be determined using \eqref{Ic} and Lemma \ref{ORR}.
\end{proof}

\subsection{Decorated graphs} Let the genus $g$ and the number of markings $n$ 
be in the stable range
\begin{align*}
    2g-2+n > 0\, .
\end{align*}

A decorated graph $\Gamma \in \mathsf{G}^{\text{Dec}}_{g,n}(3)$ consists of the data $(\mathsf{V},\mathsf{E},\mathsf{N},\gamma,\nu)$ where
\begin{enumerate}
 \item[(i)] $\mathsf{V}$ is the vertex set,
 \item[(ii)] $\mathsf{E}$ is the edge set (including possible self-edges),
 \item[(iii)] $\mathsf{N} : \{1,2,\dots,n \}\rightarrow \mathsf{V}$ is the marking assignment,
 \item[(iv)] $\mathsf{g} : \mathsf{V}\rightarrow \ZZ_{\ge 0}$ is a genus assignment satisfying
 $$g=\sum_{v\in\mathsf{V}}\mathsf{g}(v)+h^1(\Gamma) $$
 and for which $(\mathsf{V},\mathsf{E},\mathsf{N},\gamma)$ is stable graph,
 \item[(v)] $\mathsf{p} : \mathsf{V}\rightarrow \{0,1,2\}$ is an assignment to each vertex $v\in\mathsf{V}$.
\end{enumerate}

\subsection{Decomposition theorem}
By the formula
for the higher genus potential of 
Section \ref{ffxx2},
we can decompose $\mathcal{F}^{[\CC^3/\ZZ_3]}_{g}$ into contributions
of decorated graphs of genus $g$. Furthermore, we can write
the contribution corresponding to a graph $\Gamma\in \mathsf{G}^{\text{Dec}}_{g}(3)$ in terms
of vertex and edge contributions,
$$\mathcal{F}_g^{[\CC_3/\ZZ_3]} = \sum_{\Gamma \in \mathsf{G}^{\text{Dec}}_{g}(3)}
{\text{Cont}}_\Gamma\, .$$
\begin{Prop}\label{DCT}We have
 $$\text{\em Cont}_\Gamma=\frac{1}{\text{\em Aut}(\Gamma)}\sum_{\mathsf{A}\in \ZZ^{\mathsf{F}}_{>0}}\prod_{v\in\mathsf{V}}\text{\em Cont}^{\mathsf{A}}_{\Gamma}(v)\prod_{e\in\mathsf{E}}\text{\em Cont}^{\mathsf{A}}_\Gamma(e)\, ,  $$
where the vertex{\footnote{Strictly, we should have $g(\widetilde{e}_{p(v)},\widetilde{e}_{p(v)})^{1-g-\frac{n}{2}}$ in the vertex contribution,  but we shift here  $g(\widetilde{e}_{p(v)},\widetilde{e}_{p(v)})^{-\frac{n}{2}}$ to the $n$ 
incident edge contributions to be consistent with our formula for $K\PP^2$ in \cite{LP}.}} 
and edge contributions with incident flag $\mathsf{A}$-values $(a_1,a_2,...,a_n)$ and $(b_1,b_2)$ respectively are:
\begin{multline*}
\bullet    \text{\em Cont}^{\mathsf{A}}_\Gamma(v)= 
    \Bigg[\sum_{k \ge 0}\frac{\gG(\widetilde{e}_{p(v)},\widetilde{e}_{p(v)})^{1-g}}{k!}\, \cdot \ \ \ \ \ \ \ \ \ \ \ \ \      \\
    \int_{\overline{M}_{g,n+k}}\psi_1^{a_1}\dots\psi_n^{a_n}T(\psi_{n+1})\dots T(\psi_{n+k})\Bigg]\Bigg|_{t_0=t_1=0, t_{j\ge 2}=Q_{j-1\, p(v)}},
\end{multline*}
where
$Q_{kp(v)}$ is the coefficient of $z^k$ in $[(-1)^{k+1}(\mathsf{R}^{[\CC^3/\ZZ_3]}(z))^{t}\cdot \mathbf{\Psi}]_{0p(v)}$,
{{\begin{multline*}
    \bullet \text{\em Cont}^{\mathsf{A}}_{\Gamma}(e)=
    (-1)^{b_1+b_2}\, 3\, \cdot\\
    {{\left[\frac{N_{0p(v_1)}(z)N_{0p(v_2)}(w)+N_{1p(v_1)}(z) N_{2p(v_2)}(w)+N_{2p(v_1)}(z) N_{1p(v_1)}(w)}{z+w}-\frac{1}{z+w}\right]_{z^{b_1-1}w^{b_2-1}},}}
\end{multline*}}}
where $N_{ij}(z)$ is the $(i,j)$ component of $(\mathsf{R}^{[\CC^3/\ZZ_3]}(-z))^{t}\cdot\mathbf{\Psi}$.
\end{Prop}

\subsection{Legs.}
%Let $\Gamma\in\mathsf{G}^{\text{Dec}}_{g,n}(3)$ be a stable graph. While no markings %are needed to define the Gromov-Witten invariants of $[\CC^3/\ZZ_3]$, 
%
To compute the potentials $\mathcal{F}^{[\CC^3/\ZZ_3]}_{g,n}(\phi_{a_1},\ldots,\phi_{a_n})$,
the contributions of stable graphs with markings are required,
$$\mathcal{F}_{g,n}^{[\CC_3/\ZZ_3]}(\phi_{a_1,\ldots,\phi_{a_n})} = \sum_{\Gamma \in \mathsf{G}^{\text{Dec}}_{g,n}(3)}
{\text{Cont}}_\Gamma(\phi_{a_1},\ldots,\phi_{a_n})\, .$$

\begin{Prop}\label{Leg}
We have
\begin{multline*}
    \text{\em Cont}_\Gamma (\phi_{k_1},\dots,\phi_{k_n})=\\  \frac{1}{| \text{\em Aut}(\Gamma)|}\sum_{\mathsf{A}\in \mathds{Z}^{\mathsf{F}}_{>0}}\prod_{v\in\mathsf{V}}\text{\em Cont}_\Gamma^\mathsf{A}(v)\prod_{e\in\mathsf{E}}\text{\em Cont}_\Gamma^{\mathsf{A}}(e)\prod_{l\in\mathsf{L}}\text{\em Cont}_\Gamma^\mathsf{A}(l)\,,
\end{multline*}
where the leg contribution $\text{\em Cont}_\Gamma^\mathsf{A}(l)$ is given by $z^{\mathsf{A}(l)-1}$ coefficient of 
$$[(-1)^{\mathsf{A}(l)-1}(\mathsf{R}^{[\CC^3/\ZZ_3]}(z))^{t}\cdot\mathbf{\Psi}]_{\mathsf{Inv}(k_l)p(l)}\,. $$
The vertex and edge contributions are same as
in Proposition \ref{DCT}.
\end{Prop}

\subsection{Vertex, edge, and legs analysis}\label{svel}
We analyze here the vertex and edge contributions of Proposition  \ref{DCT}.

\begin{Lemma}\label{Vertex}We have
 $\text{\em Cont}^{\mathsf{A}}_\Gamma(v)\in\CC[L^{\pm 1}]$.
\end{Lemma}

\begin{proof}
 The result is a direct consequence of from Proposition \ref{DCT} and with Lemma \ref{RR}.
\end{proof}

 Let $e\in\mathsf{E}$ be an edge connecting the vertices $v_1$,$v_2 \in \mathsf{V}$. Let the $\mathsf{A}$-values of the respective half-edges be $(k,l)$.

\begin{Lemma}\label{Edge}We have $\text{\em Cont}^{\mathsf{A}}_{\Gamma}(e)\in \CC[L^{\pm 1},X]$ and
\begin{itemize}
 \item  the degree of $\text{\em Cont}^{\mathsf{A}}_\Gamma(e)$ with respect to $X$ is 1,
 \item  the coefficient of $X$ in $\text{\em Cont}^{\mathsf{A}}_\Gamma(e)$ is
 $$(-1)^{k+l}\frac{3\widetilde{P}^{k-1}_{2p(v_1)}\widetilde{P}^{l-1}_{2p(v_2)}}{L\lambda_{p(v_1)}^{k-1}\lambda_{p(v_2)}^{l-2}}\, . $$
\end{itemize}

\end{Lemma}

\begin{proof}
 The claims follow from Proposition \ref{DCT} together with Lemmas \ref{RR} and \ref{RR2}.
\end{proof}

Similarly, using the contribution formula of Proposition \ref{Leg}, we obtain the following result.

\begin{Lemma}
 The leg contributions satisfy: 
 \begin{itemize}
     \item when the insertion at the marking $l$ is $\phi_0$,
     $$\text{\em Cont}^{\mathsf{A}}_\Gamma(l)\in \CC(\lambda_0,\lambda_1,\lambda_2)[L^{\pm 1}]\,,$$
     \item when the insertion at the marking $l$ is $\phi_1$,
     $$C_1\cdot\text{\em Cont}^{\mathsf{A}}_\Gamma(l)\in \CC(\lambda_0,\lambda_1,\lambda_2)[L^{\pm 1}]\,,$$
     \item when the insertion at the marking $l$ is $\phi_2$,
     $$\frac{1}{C_1}\cdot\text{\em Cont}^{\mathsf{A}}_\Gamma(l)\in \CC(\lambda_0,\lambda_1,\lambda_2)[L^{\pm 1},X]\,.$$
 \end{itemize}
\end{Lemma}

\subsection{Proof of Theorem \ref{ooo}} By definition, we have
\begin{align}\label{ffww}
    A_2(\psit)=\frac{1}{L^3}\left(3X-3-\frac{L^3}{18}\right)\, .
\end{align}
Hence, claim (i) of Theorem \ref{ooo},
$$\mathcal{F}^{[\CC^3/\ZZ_3]}_g(\psit)\in \CC[L^{\pm 1}][A_2]\,,$$
follows from Proposition \ref{DCT} and Lemma \ref{Vertex}-\ref{Edge}.  Claim (ii), $${\text{$\mathcal{F}^{[\CC^3/\ZZ_3]}_g$ has at most degree $3g-3$ with respect to $A_2$}}\,,$$
holds since a stable graph of genus $g$ has at most $3g-3$ edges. Since
$$\frac{\partial}{\partial T}=\frac{\psit}{C_1}\frac{\partial}{\partial \psit}\,,$$
claim (iii),
\begin{align}\label{DC1}
    \frac{\partial^k \mathcal{F}^{[\CC^3/\ZZ_3]}_g}{\partial T^k}(\psit)\in\CC[L^{\pm 1}][A_2][C_1^{-1}]\,,
\end{align}
follows since the ring 
$$\CC[L^{\pm 1}][A_2]=\CC[L^{\pm 1}][X]$$
is closed under the action of the differential operator
$$\mathsf{D}=\psit\frac{\partial}{\partial \psit}$$
by Lemma \ref{Relation}. The degree of $C_1^{-1}$ in \eqref{DC1} is $1$ which yields claim (iv).  \qed

The same argument can also be applied to potentials with insertions to immediately
yield the parallel result for part (i).

\vspace{6pt}
\noindent{\bf Theorem $\mathbf{1'}$.}
{\em After the change of variables given by the inverse of mirror map,  
$$\mathcal{F}_{g,n}^{[\CC^3/\ZZ_3]}(\phi_{a_1},\ldots,\phi_{a_n}) (\psit) \in \mathbb{C}[L^{\pm1}][A_2, C_1^{\pm}]$$
for  $2g-2+n> 0$.}

\subsection{Proof of Theorem 2.}

Let $\Gamma \in \mathsf{G}^{\text{Dec}}_{g,0}(3)$ be a decorated graph. Let us fix an edge $f\in\mathsf{E}(\Gamma)$:
\begin{itemize}
 
 \item [$\bullet$]if $\Gamma$ is connected after deleting $f$, denote the resulting graph by
 $$\Gamma^0_f\in \mathsf{G}^{\text{Dec}}_{g-1,2}(3),$$
 
 \item [$\bullet\bullet$]if $\Gamma$ is disconnected after deleting $f$, denote the resulting two graphs by
 $$\Gamma^1_f\in\mathsf{G}^{\text{Dec}}_{g_1,1}(3)\ \ \  \text{and} \ \ \ \Gamma^2_f\in\mathsf{G}^{\text{Dec}}_{g_2,1}(3)$$
where $g=g_1+g_2$.

\end{itemize}

There is no canonical order for the 2 new markings. We will always sum over the 2 labellings. So more precisely, the graph $\Gamma^0_f$ in case $\bullet$ should be viewed as sum of 2 graphs
$$\Gamma^0_{f,(1,2)}+\Gamma^0_{f,(2,1)}\,.$$
Similarly, in case $\bullet\bullet$, we will sum over the ordering of $g_1$ and $g_2$. As usually, the summation will be later compensated by a factor of $\frac{1}{2}$ in the formulas.

 By Proposition \ref{DCT}, we have the following formula for the contribution of the graph $\Gamma$ to the Gromov-Witten theory of $[\CC^3/\ZZ_3]$,
  $$\text{Cont}_\Gamma=\frac{1}{\text{Aut}(\Gamma)}\sum_{\mathsf{A}\in \ZZ^{\mathsf{F}}_{>0}}\prod_{v\in\mathsf{V}}\text{Cont}^{\mathsf{A}}_{\Gamma}(v)\prod_{e\in\mathsf{E}}\text{Cont}^{\mathsf{A}}_\Gamma(e)\, .  $$

Let $f$ connect the vertices $v_1,v_2\in \mathsf{V}(\Gamma)$. Let the $\mathsf{A}$-values of the respective half-edges be $(k,l)$. By Lemma \ref{Edge}, we have
\begin{align}\label{EA}
    \frac{\partial \text{Cont}^{\mathsf{A}}_\Gamma(f)}{\partial X}=(-1)^{k+l}\frac{3\widetilde{P}^{k-1}_{2 p(v_1)}\widetilde{P}^{l-1}_{2 p (v_2)}}{L\lambda_{p(v_1)}^{k-2}\lambda_{p(v_2)}^{l-2}}\, .
\end{align}

\begin{itemize}
 \item If $\Gamma$ is connected after deleting $f$, we have
 \begin{align*}
     \frac{1}{|\text{Aut}(\Gamma)|}\sum_{\mathsf{A}\in \ZZ^{\mathsf{F}_{\ge 0}}}\left(\frac{L^3}{3C_1^2}\right)\frac{\partial \text{Cont}^{\mathsf{A}_\Gamma(f)}}{\partial X}\prod_{v\in\mathsf{V}}\text{Cont}^{\mathsf{A}_\Gamma(v)}\prod_{e\in\mathsf{E}}\text{Cont}^{\mathsf{A}}_\Gamma(e)\\
     =\frac{1}{2}\text{Cont}_{\Gamma^0_f}(\phi_1,\phi_1)\, .
 \end{align*}
\end{itemize} 
The derivation is simply by applying \eqref{EA} on the left and Proposition \ref{Leg} on the right.
\begin{itemize} 
 \item If $\Gamma$ is disconnected after deleting $f$, we obtain
 \begin{align*}
     \frac{1}{|\text{Aut}(\Gamma)|}\sum_{\mathsf{A}\in \ZZ^{\mathsf{F}_{\ge 0}}}\left(\frac{L^3}{3C_1^2}\right)\frac{\partial \text{Cont}^{\mathsf{A}_\Gamma(f)}}{\partial X}\prod_{v\in\mathsf{V}}\text{Cont}^{\mathsf{A}_\Gamma(v)}\prod_{e\in\mathsf{E}}\text{Cont}^{\mathsf{A}}_\Gamma(e)\\
     =\frac{1}{2}\text{Cont}_{\Gamma^1_f}(\phi_1)\text{Cont}_{\Gamma^2_f}(\phi_1)
 \end{align*}

\end{itemize}
by the same method.
 By combining the above two equations for all the edges of all the graphs $\Gamma \in \mathsf{G}^{\text{Dec}}_{g}(3)$ and using the vanishing
 $$\frac{\partial \text{Cont}^{\mathsf{A}}_\Gamma(v)}{\partial X}=0 $$ of Lemma \eqref{Vertex}, we obtain
 \begin{multline}\label{EIP}
     \left(\frac{L^3}{3C_1^2}\right) \frac{\partial}{\partial X}\lan\lan\ran\ran^{[\CC^3/\ZZ_3]}_{g}=\\
     \frac{1}{2}\sum_{i=1}^{g-1}\lan\lan\phi_1\ran\ran^{[\CC^3/\ZZ_3]}_{g-i,1}\lan\lan\phi_1\ran\ran^{[\CC^3/\ZZ_3]}_{i,1}+\frac{1}{2}\lan\lan\phi_1,\phi_1\ran\ran^{[\CC^3/\ZZ_3]}_{g-1,2}.
 \end{multline}
We have followed here the notation of Section \ref{vvv234}. The equality \eqref{EIP} holds in the ring $\CC[L^{\pm 1}][A_2,C_1^{-1}]$.

Since $A_2=\frac{1}{L^3}\left(3X-3-\frac{L^3}{18}\right)$ and $\lan \lan\,\ran \ran^{[\CC^3/\ZZ_3]}_g=\mathcal{F}^{[\CC^3/\ZZ_3]}_g$, the left side of \eqref{EIP} is, by the chain rule,
$$\frac{1}{C_2}\frac{\partial \mathcal{F}^{[\CC^3/\ZZ_3]}_g}{\partial A_2}\in\CC[L^{\pm 1}][A_2,C_1^{-1}]\,.$$
On the right side of \eqref{EIP}, we have
$$\lann \phi_1\rann^{[\CC^3/\ZZ_3]}_{g-i,1}=\frac{\partial \mathcal{F}^{[\CC^3/\ZZ_3]}_{g-i}}{\partial T}\in \CC[[\psit]]\,.$$
Similarly, we obtain
\begin{eqnarray*}
    \lann \phi_1\rann^{[\CC^3/\ZZ_3]}_{i,1}&=&\frac{\partial \mathcal{F}^{[\CC^3/\ZZ_3]}_i}{\partial T}\in \CC[[\psit]]\,,\\
    \lann \phi_1,\phi_1\rann^{[\CC^3/\ZZ_3]}_{g-1,2}&=&\frac{\partial^2 \mathcal{F}^{[\CC^3/\ZZ_3]}_{g-1}}{\partial T^2}\in \CC[[\psit]]\,.
\end{eqnarray*}
Together, the above equations transform \eqref{EIP} into exactly the holomorphic anomaly equation of Theorem \ref{ttt},
$$\frac{1}{C_1^2}\frac{\partial \mathcal{F}_g^{[\CC^3/\ZZ_3]}}{\partial{A_2}}
= \frac{1}{2}\sum_{i=1}^{g-1} 
\frac{\partial \mathcal{F}_{g-i}^{[\CC^3/\ZZ_3]}}{\partial{T}}
\frac{\partial \mathcal{F}_i^{[\CC^3/\ZZ_3]}}{\partial{T}}
+
\frac{1}{2}
\frac{\partial^2 \mathcal{F}_{g-1}^{[\CC^3/\ZZ_3]}}{\partial{T}^2}\,
,$$
as an equality in $\CC[[\psit]]$.

The series $L$ and $A_2$ are expected to be algebraically independent. Since we do not have a proof of the independence, to lift holomorphic anomaly equation to the equality
$$\frac{1}{C_1^2}\frac{\partial \mathcal{F}_g^{[\CC^3/\ZZ_3]}}{\partial{A_2}}
= \frac{1}{2}\sum_{i=1}^{g-1} 
\frac{\partial \mathcal{F}_{g-i}^{[\CC^3/\ZZ_3]}}{\partial{T}}
\frac{\partial \mathcal{F}_i^{[\CC^3/\ZZ_3]}}{\partial{T}}
+
\frac{1}{2}
\frac{\partial^2 \mathcal{F}_{g-1}^{[\CC^3/\ZZ_3]}}{\partial{T}^2}$$
in the ring $\CC[L^{\pm 1}][A_2,C_1^{-1}]$, we must prove the equalities
\begin{equation*}
    \lann \phi_1\rann^{[\CC^3/\ZZ_3]}_{g-i,1}=\frac{\partial \mathcal{F}^{[\CC^3/\ZZ_3]}_{g-i}}{\partial T}\,,\,\,\,\, \, \lann \phi_1\rann^{[\CC^3/\ZZ_3]}_{i,1}=\frac{\partial \mathcal{F}^{[\CC^3/\ZZ_3]}_i}{\partial T}\in \CC[[\psit]]\,,
\end{equation*}
$$   \lann \phi_1,\phi_1\rann^{[\CC^3/\ZZ_3]}_{g-1,2}=\frac{\partial^2 \mathcal{F}^{[\CC^3/\ZZ_3]}_{g-1}}{\partial T^2}\in \CC[[\psit]]$$
hold in the ring $\CC[L^{\pm 1}][A_2,C_1^{-1}]$. The lifting follow from the argument in Section 7.3 in \cite{KL}.

 We do not study the genus $1$ unpointed series $\mathcal{F}^{[\CC^3/\ZZ_3]}$ in the paper, so we take
 \begin{align*}
     \lann \phi_1\rann^{[\CC^3/\ZZ_3]}_{g-i,1}&=\frac{\partial \mathcal{F}^{[\CC^3/\ZZ_3]}_{g-i}}{\partial T}\,,\\
     \lann \phi_1\rann^{[\CC^3/\ZZ_3]}_{i,1}&=\frac{\partial \mathcal{F}^{[\CC^3/\ZZ_3]}_i}{\partial T}\in \CC[[\psit]]
 \end{align*}
as definitions of the right side in the genus $1$ case. There is no difficulty in calculating these series explicitly using Proposition \ref{Leg},
\begin{align*}
    \frac{\partial \mathcal{F}^{[\CC^3/\ZZ_3]}_1}{\partial T}&=\frac{1}{18C_1}L^3A_2\,,\\
    \frac{\partial^2 \mathcal{F}^{[\CC^3/\ZZ_3]}_1}{\partial T^2}&=\frac{1}{C_1}\mathsf{D}\left(\frac{1}{18C_1}L^3A_2\right)\,.
\end{align*}

\subsection{Bounding the degree}

 The degrees in $L$ of the terms of
 $$\mathcal{F}^{[\CC^3/\ZZ_3]}_{g}\in \mathbb{C}[L^{\pm1}][A_2]$$ for $[\CC^3/\ZZ_3]$ 
 always fall in the range
 \begin{equation}\label{ds33}
 [9-9g, 6g-6]\, .
 \end{equation}
In particular, the constant (in $A_2$) term of $\mathcal{F}^{[\CC^3/\ZZ_3]}_{g}$
missed by the
holomorphic anomaly equation for $[\CC^3/\ZZ_3]$ is a
Laurent polynomial in $L$ with degrees in the range
\eqref{ds33}.
The bound \eqref{ds33} is a consequence of Proposition \ref{DCT}, 
 the vertex and edge analysis of Section \ref{svel}, and the
 following result.
 
\begin{Lemma}
 The degrees in $L$ of $\widetilde{P}^k_{ij}$ fall in the range
 \begin{align*}
     [-i,2k]\, .
 \end{align*}
\end{Lemma} 

\begin{proof}
 The proof for the functions $\widetilde{P}^k_{0j}$ follows from the arguments of \cite{ZaZi}. The proof for $\widetilde{P}^k_{1j}$ and $\widetilde{P}^k_{2j}$ follows  from Lemma \ref{RR2}.
\end{proof}

For $\mathcal{F}^{[\CC^3/\ZZ_3]}_{2}$ (resp. $\mathcal{F}^{[\CC^3/\ZZ_3]}_{3}$),  the $L$ degrees can be seen to vary between
 $0$ and $6$ (resp. $0$ and $12$) in the formula in Section \ref{EC} when
 rewritten in terms of $A_2$ using \eqref{ffww}.
 The sharper range
\begin{equation*}
 [0, 6g-6]\, 
 \end{equation*}
 proposed in \cite{ASYZ} 
for the $L$ degrees of $\mathcal{F}^{[\CC^3/\ZZ_3]}_{g}$
is found in examples. How to derive the sharper bound
from properties of the functions $\widetilde{P}^k_{ij}$ is
an interesting question.

\section{Crepant resolution correspondence}

\subsection{$\mathbf{R}$-matrix of $K\PP^2$}
For Gromov-Witten theories
in the torus equivariant setting, Givental proved a reconstruction result in the semisimple case
using the localization of the virtual class \cite{GrP}. We have applied the method to the stable quotient theory of local $\PP^2$ in \cite{LP}. The results are summarized here.

Let  $H$ be the hyperplane class in $H_{\mathsf{T}}^*(\PP^2)$, and let
\begin{equation}\label{dpp2}
\{1, H, H^2\}\in H_\mathsf{T}^*(\PP^2) 
\end{equation}
be a basis.
The inner product $\mathbf{g}^{K\PP^2}$ in the basis 
\eqref{dpp2} is given by:
  \begin{align}\label{IP2}
   \mathbf{g}^{K\PP^2}=-\frac{1}{3}
  \left[ {\begin{array}{ccc}
   1         & 0 & 0 \\
   0         & 0 & 1 \\
   0         & 1 & 0 \\
  \end{array} } \right]\,. 
\end{align} 

%Following the notation of \cite{LP}, we define series for the $K\PP^2$ %geometry.

For $\gamma\in H^*_\mathsf{T}(K\PP^2)$, we define a $q$-series $\overline{\mathds{S}}_i$ using quasimap invariants of $K\PP^2$ by
\begin{align*}
    \overline{\mathds{S}}
    %^{K\PP^2}
    _i(\gamma)=e^{K\PP^2}_i\left \lan \frac{\phi^{K\PP^2}_i}{z-\psi},\gamma\right \ran^{K\PP^2}_{0,2}\,.
\end{align*}
We follow here the notation{\footnote{In particular, the formulas
hold after the specialization \eqref{specc} of equivariant
parameters.}} of \cite[Section 3]{LP} 
where $$e^{K\PP^2}_i=-3(1-\lambda_i)(1-\lambda_i^2)=-9$$ 
is the equivariant Euler class of the tangent space of $K\PP^2$ at the 
fixed point $p_i$, and $\phi^{K\PP^2}_i$ is the canonical basis element
$$\phi^{K\PP^2}_i=\frac{-3\lambda_i\prod_{j\ne i}(H-\lambda_j)}{e^{K\PP^2}_i}\,.$$

The following asymptotic form{\footnote{The notation $\widetilde{P}^{k,K\PP^2}_{i0}$ differs slightly from $R_{ik}$ in \cite{LP}. More precisely, the correspondence of the index $i$ between $\widetilde{P}^{k,K\PP^2}_{i0}$ and $R_{ik}$ in \cite{LP} is $\{0,1,2\}\rightarrow\{0,2,1\}$. The difference will not have any effect in the higher genus formula.}} 
 of the series $\overline{\mathds{S}}_j(H^i)$ plays
a crucial role
\cite[Section 3.4]{LP}:
\begin{align}\label{ASKP2}
    \nonumber\overline{\mathds{S}}_j(1)=e^{\frac{\mu \zeta_j}{z}} \left( \widetilde{P}_{00}^{0,K\PP^2}+\widetilde{P}_{00}^{1,K\PP^2}(\frac{z}{\zeta^j})+\widetilde{P}_{00}^{2,K\PP^2}(\frac{z}{\zeta^j})^2+\dots\right)\, ,\\
    \overline{\mathds{S}}_j(H)=e^{\frac{\mu \zeta_j}{z}} \frac{L^{K\PP^2} \zeta^{j}}{C^{K\PP^2}_1}\left( \widetilde{P}_{20}^{0,K\PP^2}+\widetilde{P}_{20}^{1,K\PP^2}(\frac{z}{\zeta^j})+\widetilde{P}_{20}^{2,K\PP^2}(\frac{z}{\zeta^j})^2+\dots\right)\, ,\\
    \nonumber\overline{\mathds{S}}_j(H^2)=e^{\frac{\mu \zeta_j}{z}} \frac{(L^{K\PP^2})^2 \zeta^{2j}}{C^{K\PP^2}_1 C^{K\PP^2}_2}\left( \widetilde{P}_{10}^{0,K\PP^2}+\widetilde{P}_{10}^{1,K\PP^2}(\frac{z}{\zeta^j})+\widetilde{P}_{10}^{2,K\PP^2}(\frac{z}{\zeta^j})^2+\dots\right)\, ,
\end{align}
for $0\leq j\leq 2$.
Here, $\mu(q)=\int_{0}^q (L^{K\PP^2}(x)-1)\frac{dx}{x}$.
Define
$$X^{K\PP^2}=\frac{\mathsf{D}^{K\PP^2}C^{K\PP^2}_1}{C^{K\PP^2}_1}\,,$$
where $\mathsf{D}^{K\PP^2}=q\frac{d}{dq}\,.$
The series $\widetilde{P}_{ij}^{k,K\PP^2}$ satisfy following system of equations for $j=0, 1, 2$:
\begin{align}\label{RP2}
    \nonumber \widetilde{P}_{2j}^{k+1,K\PP^2}&=\widetilde{P}_{0j}^{k+1,K\PP^2}+\frac{\mathsf{D}^{K\PP^2}\widetilde{P}_{0j}^{k,K\PP^2}}{L^{K\PP^2}}\,,\\
    \widetilde{P}_{1j}^{k+1,K\PP^2}&=\widetilde{P}_{2j}^{k+1,K\PP^2}+\frac{\mathsf{D}^{K\PP^2}\widetilde{P}_{2j}^{k,K\PP^2}}{L^{K\PP^2}}+\left(\frac{\mathsf{D}^{K\PP^2}L^{K\PP^2}}{(L^{K\PP^2})^2}-\frac{X^{K\PP^2}}{L^{K\PP^2}}\right)\widetilde{P}_{2j}^{k,K\PP^2}\,,\\
    \nonumber \widetilde{P}_{0j}^{k+1,K\PP^2}&=\widetilde{P}_{1j}^{k+1,K\PP^2}+\frac{\mathsf{D}^{K\PP^2}\widetilde{P}_{1j}^{k,K\PP^2}}{L^{K\PP^2}}-\left(\frac{\mathsf{D}^{K\PP^2}L^{K\PP^2}}{(L^{K\PP^2})^2}-\frac{X^{K\PP^2}}{L^{K\PP^2}}\right)\widetilde{P}_{1j}^{k,K\PP^2}\,,
\end{align}
with the initial conditions
\begin{align*}
    &\widetilde{P}_{ij}^{0,K\PP^2}|_{q=0}=1\,,\\
    &\widetilde{P}_{ij}^{k,K\PP^2}|_{q=0}=0\,\,\,\,\text{for}\,\,k \ge 1\,.
\end{align*}

Denote by $\mathsf{Q}(z)$ the matrix with the coefficient of $z^k$ in $(i,j)$
component $P^{k,K\PP2}_{ij}$ are defined by the following equations for $k\ge 0$ and $j=0,1,2$:
\begin{align*}
    \nonumber\widetilde{P}^{k,K\PP^2}_{00} & = P^{k,K\PP2}_{0j} \zeta^{kj}\\
    \widetilde{P}^{k,K\PP^2}_{10} & = \frac{L}{C_1 \zeta^{2j}}P^{k,K\PP^2}_{1j} \zeta^{kj}\\
    \nonumber\widetilde{P}^{k,K\PP2}_{20} & = \frac{C_1}{L \zeta^{j}} P^{k,K\PP2}_{2j}\zeta^{kj}\, .
\end{align*} 
Define a new matrix $\widetilde{\mathsf{R}}^{K\PP^2}(z)$ by
$$\widetilde{\mathsf{R}}^{K\PP^2}(z)=\mathbf{\Psi}\cdot\mathsf{Q}(z),$$
where
$$\mathbf{\Psi}^{K\PP^2}=\frac{-i}{3}
  \left[ {\begin{array}{ccc}
    1 & \frac{L^{K\PP^2}}{C^{K\PP^2}_1}        & \frac{C^{K\PP^2}_1}{L^{K\PP^2}} \\
    1 & \zeta \frac{L^{K\PP^2}}{C^{K\PP^2}_1}  & \zeta^2\frac{C^{K\PP^2}_1}{L^{K\PP^2}} \\
    1 & \zeta^2\frac{L^{K\PP^2}}{C^{K\PP^2}_1} & \zeta\frac{C^{K\PP^2}_1}{L^{K\PP^2}} \\
  \end{array} } \right]\,.$$

Define a new diagonal matrix
$$\mathsf{B}^{K\PP^2}(z)= \text{Diag}\, \big[B^{K\PP^2}_0(z),B^{K\PP^2}_1(z),B^{K\PP^2}_2(z)\big]$$
where
$$B_j^{K\PP^2}(z)={\rm Exp}\left(-\sum_{k=1}^\infty \frac{N_{2k-1,j}}{2k-1}\frac{B_{2k}(0)}{2k} \left(\frac{z}{\zeta^i}\right)^{2k-1}\right)\,$$
and
$N_{k,j}=(-\frac{1}{3\zeta^j})^k+\sum_{l=1}^2 (\frac{1}{\zeta^{j}-\zeta^{j+l}})^k$.

The $\mathbf{R}$-matrix $\mathsf{R}^{K\PP^2}$ in the normalized canonical basis for $K\PP^2$ is given by the following result, see \cite{SS, LP}.

\begin{Prop}\label{RM1} We have 
\begin{align*}
    \left[\mathsf{R}^{K\PP^2}(z)\right]_{ij}=\left[\widetilde{\mathsf{R}}^{K\PP^2}(z)\right]_{ij}\cdot \mathsf{B}^{K\PP^2}(z)\,.
\end{align*}

\end{Prop}

\subsection{Proof of Theorem \ref{crc}}
The $\mathbf{R}$-matrix approach to Theorem \ref{crc} will establish
a more general results for potentials with insertions.
Let
\begin{equation*} 
\mathcal{F}^{K\PP^2}_{g,n}(H^{a_1},\ldots,H^{a_n}) = 
  \sum_{d=0}^\infty {q^d} \int_{[\overline{M}_{g,n}(K\proj^2,d)]^{vir}}\prod_{k=1}^n{\text{ev}}_i^*(H^{a_k})
  \,,
 \end{equation*}
 be the Gromov-Witten potential for $K\PP^2$.
Define 
the map 
$$\iota :(H^*_{\mathsf{T},\text{orb}}([\CC^3/\ZZ_3]),\mathbf{g}) \rightarrow (H^*_{\mathsf{T}}(K\PP^2),\mathbf{g}^{K\PP^2})$$
by the rule
\begin{align*}
    \iota(\phi_0)=1\,,\,\,\, \iota(\phi_1)=H\,,\,\,\,\iota(\phi_2)=H^2\,.
\end{align*}
By Theorem $1'$ and \cite[Theorem 1]{LP} respectively, we have 
\begin{eqnarray*}
\mathcal{F}^{[\CC^3/\ZZ_3]}_{g,n}(\phi_{a_1},\ldots,\phi_{a_n}) & \in &
\mathbb{C}[L^{\pm1}][X, C_1^{\pm1}]\, ,
\\
\mathcal{F}^{K\PP^2}_{g,n}(H^{a_1},\ldots,H^{a_n}) 
&\in &\mathbb{C}[(L^{K\PP^2})^{\pm 1}][X^{K\PP^2}, (C_1^{K\PP^2})^{\pm 1}]\, .
\end{eqnarray*}
The following result specializes to Theorem \ref{crc}
in case there are no insertions.

\vspace{6pt}\noindent {\bf Theorem $\mathbf{4'}$.}
{\em For $g$ and $n$ in the stable range, the crepant resolution
 correspondence 
 \begin{align*}
 \mathcal{F}^{[\CC^3/\ZZ_3]}_{g,n}(\phi_{i_1},\ldots,\phi_{i_n})=(-1)^{2g-2+n}\cdot
     \mathsf{P}\left(\mathcal{F}^{K\PP^2}_{g,n}(\iota(\phi_{i_1}),\ldots,\iota(\phi_{i_n}))\right)
     \,
 \end{align*}
    holds with the ring homomorphism
    $$\mathsf{P}:\mathbb{C}[(L^{K\PP^2})^{\pm 1}][X^{K\PP^2}, (C_1^{K\PP^2})^{\pm 1}]\rightarrow
    \mathbb{C}[L^{\pm1}][X, C_1^{\pm1}]$$
    defined by
    $$
\mathsf{P}(L^{K\PP^2})=-\frac{L}{3}\,,\ \
    \mathsf{P}(X^{K\PP^2})=-\frac{X}{3}\,, \ \ 
        \mathsf{P}(C_1^{K\PP^2}) = \frac{1}{3} C_1\, .
$$}
\vspace{6pt}

\begin{proof}
The first step is to prove that the map $\iota$ matches the 
CohFT 
structures in genus 0 (up to sign). Certainly, $\iota$
preserves pairings up to a sign,
\begin{align}\label{SC}
    \mathbf{g}=-\mathbf{g}^{K\PP^2}\,.
\end{align}
Using the result of \cite[Section 5]{LP}, we easily obtain the following genus 0 results for $K\PP^2$:
\begin{align*}
    &\lan 1,1,1\ran_{0,3}^{K\PP^2}=-\frac{1}{3}\, ,   
    &\lan1,H,H^2\ran_{0,3}^{K\PP^2}=-\frac{1}{3}\, ,\\
    & \lan H,H,H\ran_{0,3}^{K\PP^2}=-\frac{1}{3}\frac{(L^{K\PP^2})^3}{(C_1^{K\PP^2})^3}\, ,   & \lan H^2,H^2,H^2\ran_{0,3}^{K\PP^2}=-\frac{1}{3}\frac{(C_1^{K\PP^2})^3}{(L^{K\PP^2})^3}\, .
\end{align*}
For other choices of insertions, the $3$-point functions
in genus 0 vanish.
The genus $0$ invariants for $K\PP^2$ match Lemma \ref{g0C3} via the ring homomorphism $\mathsf{P}$ and the map $\iota$ up to the sign $(-1)^{2g-2+n}$.

%***Must say something here about the genus 0 theories, and thus the topological CohFTs mathching with respect to $\iota$. But this involves a sign, right? It is best to write the formula for the genus 0 CohFT match with the explicit sign and cite CCIT.**

The second step is to match the $\mathbf{R}$-matrices of the two CohFTs. 
The two system of equations \eqref{RC3} and \eqref{RP2} are equivalent via the transformation $\mathsf{P}$ defined by \eqref{RH} up to another sign change. The effect
of the 
latter sign change cancels the former 
sign change \eqref{SC} in the higher genus formula. More precisely, the sign changes will contribute the
global factors $(-1)^{1-g}$ and $(-1)^{3g-3+n}$ respectively in the genus $g$, $n$-marked Gromov-Witten potential function.

Therefore, to prove Theorem \ref{crc}, we must only  match the constant terms of the $\mathbf{R}$-matrices.
%The result will then prove Theorem \ref{crc}
%for the full CohFT potentials, 
%%\begin{equation*} 
%\mathcal{F}^{K\PP^2}_{g,n}(H^{a_1},\ldots,H^{a_n}) = 
%  \sum_{d=0}^\infty {q^d} %\int_{[\overline{M}_{g,n}(K\proj^2,d)]^{vir}}\prod_{k=1}^n{\text{ev}}_i%^*(H^{a_k})
%  \,,
% \end{equation*}
% 
%\begin{multline*} 
%\mathcal{F}^{[\CC^3/\ZZ_3]}_{g,n}(\phi_{a_1},\ldots,\phi_{a_n})
%= \\
%  \sum_{d=0}^\infty \frac{\Psit^d}{d!} \int_{[\overline{M}^{\mathsf{orb}}_{g,n+d}([\CC^3/\ZZ_3],0)]^{vir}}\prod_{k=1}^n{\text{ev}}_i^*(\phi_{a_k})\prod_{i=n+1}^{n+d}
% {\text{ev}}_i^*(\phi_1)\,.
% \end{multline*}
%and $\mathcal{F}^{[\CC^3/\ZZ_3]}_{g,n}(\phi_{a_1},\ldots,\phi_{a_n})$
%for all $g$ and $n$ in the the stable range,
%\begin{align*}
%     \mathsf{P}\left(\mathcal{F}^{K\PP^2}_{g,n}(\iota(\phi_{i_1}),\ldots,\iota(\phi_{i_n}))\right)=
%     \mathcal{F}^{[\CC^3/\ZZ_3]}_{g,n}(\phi_1,\ldots,\phi_n)\,.
% \end{align*}
We 
apply the result of \cite{ZaZi} to \eqref{RP2}, to conclude
\begin{align*}
    %\label{R1}
    &\widetilde{P}_{i0}^{k,K\PP^2}\in\CC[L^{K\PP^2}]\,\,\,\text{for}\,\, i=0,2\,, \\
    %\label{R2}
    &\widetilde{P}_{10}^{k,K\PP^2}\in\CC[(L^{K\PP^2})^{\pm 1},X^{K\PP^2}]\,.
\end{align*}
Denote by $a_{ik}$ the constant term in the Laurent series of $\widetilde{P}_{i0}^{k,K\PP^2}$ in $L^{K\PP^2}$. From \eqref{RP2}, we can prove $a_{ik}$ is independent of $X^{K\PP^2}$. Therefore, we have $$a_{ik}\in\QQ\, 
\  \text{for}\ i=0,1,2 \, .$$  
The last step in the proof of Theorem \ref{crc} is
the following Lemma proven in the Appendix by T. Coates
and H. Iritani.

\begin{Lemma}\label{CI} The equality
of power series in $z$,
 \begin{multline*}
    {\rm Exp}\left(-\sum_{k=1}^\infty \frac{N_{2k-1,0}}{2k-1}\frac{B_{2k}(0)}{2k} z^{2k-1}\right)\sum_{k=0}^\infty a_{\mathsf{Inv}(i)k} z^k =\\ {\rm Exp}\left(3\sum_{k=1}^\infty(-1)^{k+1} \frac{B_{3k+1}(i/3)}{3k+1}\frac{z^{3k}}{3k}\right)\,,
\end{multline*}
holds for $i=0,1,2$.
\end{Lemma}

The left side of equality of Lemma \ref{CI} for $i=0$ 
%(resp. $i=1$, $i=2$) 
computes the constant term 
with respect to $L^{K\PP^2}$ 
of the coefficients  of the first row 
%(resp. second row after multiplication by $\frac{L^{K\PP^2}}{C_1^{K\PP^2}}$, third row after multiplication by $\frac{C_1^{K\PP^2}}{L^{K\PP^2}}$) 
of the $\mathbf{\Psi}^{-1}\mathbf{R}$-matrix of $K\PP^2$
by Proposition \ref{RM1}.
For $i=1$ and $i=2$, the left side of Lemma \ref{CI}
computes the constant terms in $L^{K\PP^2}$  of the coefficients
of the second row and the third rows (after multiplication{\footnote{Both
$\frac{L^{K\PP^2}}{C_1^{K\PP^2}}$ and $\frac{C_1^{K\PP^2}}{L^{K\PP^2}}$ 
have constant term in $q$ equal to 1.}}
by $\frac{L^{K\PP^2}}{C_1^{K\PP^2}}$ and $\frac{C_1^{K\PP^2}}{L^{K\PP^2}}$ 
respectively). 

Similarly, the right side of Lemma \ref{CI} for $i=0$ 
%(resp. $i=1$, $i=2$)
computes the constant term  with respect to $L$ of the coefficients of
the first row 
%(resp. second row after multiplication by $-\frac{L}{C_1}$, third row after multiplication by $-\frac{C_1}{L}$) of 
of the $\mathbf{\Psi}^{-1}\mathbf{R}$-matrix of 
$[\CC^3/\ZZ_3]$ by Proposition \ref{ORR}. 
For $i=1$ and $i=2$, the right side of Lemma \ref{CI}
computes the constant terms in $L$  of the coefficients
of the second row and the third rows (after multiplication{\footnote{Both
$-\frac{L}{C_1}$ and $-\frac{C_1}{L}$ 
have constant term in $\psit$ equal to 1.}}
by $-\frac{L}{C_1}$ and -$\frac{C_1}{L}$ 
respectively). 

Since the
constant terms match by Lemma \ref{CI}, the $\mathbf{R}$-matrix (or, equivalently,
the $\mathbf{\Psi}^{-1}\mathbf{R}$-matrix) of $[\CC^3/\ZZ_3]$ exactly equals the $\mathbf{R}$-matrix (or, equivalently, the $\mathbf{\Psi}^{-1}\mathbf{R}$-matrix) of $K\PP^2$ via the transformation $\mathsf{P}$. 
\end{proof}

\section{Calculations in low genus}\label{EC}
 We present here the
 formula for the potential function in  genus 2 and 3
 for  $[\CC^3/\ZZ_3]$  and $K\PP^2$ obtained via the $\mathbf{R}$-matrix 
 method of Section \ref{haes}. In genus 2, we have
  \begin{eqnarray*}
     \mathcal{F}^{[\CC^3/\ZZ^3]}_2&=& \frac{-291600 - 25893 L^3 - 784 L^6 - 8 L^9}{ 466560 L^3}\\ & & + \left(\frac{1}{9} + \frac{15}{8 L^3} + \frac{13 L^3}{7776}\right) X \\
     & &+ \left(-\frac{1}{18} - \frac{15}{
    8 L^3}\right) X^2 + \frac{5 X^3}{8 L^3}\,,
\end{eqnarray*}
 \begin{eqnarray*}
     \mathcal{F}^{K\PP^2}_2&=&\frac{400-959\tilde{L}^3+784\tilde{L}^6-216\tilde{L}^9}{17280\tilde{L}^3}\\
     &&+\left(-\frac{1}{3}+\frac{5}{24\tilde{L}^3}+\frac{13\tilde{L}^3}{96}\right)\tilde{X}\\
     & & +\left(-\frac{1}{2}+\frac{5}{8\tilde{L}^3}\right)\tilde{X}^2+\frac{5}{8\tilde{L}^3}\tilde{X}^3.
     \end{eqnarray*}
To simplify the formulas, we have used the notation
$\tilde{L}=L^{K\PP^2}$ and  $\tilde{X}=X^{K\PP^2}$ for $K\PP^2$.

\pagebreak

The formula for $\mathcal{F}^{[\CC^3/\ZZ^3]}_3$  is
much more complicated:
{\footnotesize
{\begin{align*}
     &\frac{26784626400 + 7043364720 L^3 + 767774781 L^6 + 44032896 L^9 + 
  1398288 L^{12} + 23328 L^{15} + 160 L^{18}}{
 9523422720 L^6}\\ \\
     &+ \frac{(-318864600 - 66331710 L^3 - 5521446 L^6 - 
    228393 L^9 - 4681 L^{12} - 38 L^{15}) X}{
 18895680 L^6} \\ \\
     &+ \frac{(531441000 + 83980800 L^3 + 4996566 L^6 + 
    132147 L^9 + 1307 L^{12}) X^2}{
 12597120 L^6} \\ \\
     &+ \frac{(-47239200 - 5318784 L^3 - 200772 L^6 - 
    2539 L^9) X^3}{
 839808 L^6} \\ \\
     &+ \left(\frac{35}{648} + \frac{675}{16 L^6} + \frac{289}{96 L^3}
     \right) X^4 
     - \frac{
 5 (324 + 11 L^3) X^5}{96 L^6} + \frac{45 X^6}{16 L^6}.
 \end{align*}}}
For $\mathcal{F}^{K\PP^2}_3$,  we
have:
{\footnotesize{
 \begin{align*}
&     \frac{16800 - 119280 \tilde{L}^3 + 351063 \tilde{L}^6 - 543616 \tilde{L}^9 + 466096 \tilde{L}^{12} - 
  209952 \tilde{L}^{15} + 38880 \tilde{L}^{18}}{
 4354560 \tilde{L}^6}\\ \\
     &+ \frac{(600 - 3370 \tilde{L}^3 + 7574 \tilde{L}^6 - 8459 \tilde{L}^9 + 4681 \tilde{L}^{12} - 
    1026 \tilde{L}^{15}) \tilde{X}}{
 8640 \tilde{L}^6}\\ \\ 
     &+ \frac{(3000 - 12800 \tilde{L}^3 + 20562 \tilde{L}^6 - 14683 \tilde{L}^9 + 
    3921 \tilde{L}^{12}) \tilde{X}^2}{
 5760 \tilde{L}^6} \\ \\
     &+ \frac{(2400 - 7296 \tilde{L}^3 + 7436 \tilde{L}^6 - 2539 \tilde{L}^9) \tilde{X}^3}{
 1152 \tilde{L}^6}\\ \\
     &+ \left(\frac{35}{8} + \frac{75}{16 \tilde{L}^6} - \frac{289}{32 \tilde{L}^3}\right) \tilde{X}^4 - \frac{
 15 (-12 + 11 \tilde{L}^3) \tilde{X}^5}{32 \tilde{L}^6} + \frac{45 \tilde{X}^6}{16 \tilde{L}^6}\, .
 \end{align*}}}
 
\noindent As stated in Theorem \ref{crc}, the above potentials match after
the ring homomorphism $\mathsf{P}$,
$$
\mathsf{P}(L^{K\PP^2})=-\frac{L}{3}\,,\ \ \
    \mathsf{P}(X^{K\PP^2})=-\frac{X}{3}\,.
$$

In \cite{BoCa}, the constant terms of $\mathcal{F}_2^{[\CC^3/\ZZ^3]}$
and $\mathcal{F}_3^{[\CC^3/\ZZ^3]}$ were computed directly
by studying the geometry of the moduli space of curves. Our calculations
agree with their results.

\pagebreak
\pagestyle{plain}

\appendix

\section{The $\mathbf{R}$-matrix identity}
\vspace{8pt}
\begin{center}
    {\em by Tom Coates and Hiroshi Iritani
    {\let\thefootnote\relax\footnote{{H.I. thanks Atsushi Kanazawa for inviting Hyenho Lho to Kyoto and for 
providing an occasion to discuss the identity in the Appendix and 
the Crepant Resolution Conjecture.}}}
    }
\end{center}
\vspace{8pt}

\subsection{Overview}
We will prove Lemma \ref{CI} by analyzing the oscillatory integrals occurring in Givental's equivariant mirror \cite{Gequiv}. We briefly recall the so-called {\em saddle point method} for finding their asymptotic behaviour, see \cite[Section 6.2]{CCIT2}. Let $f(t),\,g(t)$ be holomorphic functions on $\CC^n$ and consider the oscillatory integral
$$\int_{\Gamma} e^{f(t)/z}g(t)dt^1\dots dt^n\, ,$$
where the real $n$-dimensional cycle $\Gamma\subset\CC^n$ is chosen so that the integral converges. Let $t_0$ be a non-degenerate critical point of $f$ and choose $\Gamma$ to be the stable manifold of the Morse function 
$$t\rightarrow \mathscr{R}(f(t))$$ associated with $t_0$ 
(the union of downward gradient trajectories converging to $t_0$). Here we assume $z<0$ and study the asymptotic behaviour of the integral as $z$ approaches zero from the negative real axis. The asymptotic behaviour as $z\rightarrow 0$ is determined only by the integrand around the critical point $t_0$. We expand the integrand $e^{f(t)/z}g(t)$ in Taylor series at $t_0$ and perform termwise integration with respect to the Gaussian measure
$$e^{\frac{1}{2z}\sum_{i,j}h_{i,j}(t^i-t^i_0)(t^j-t^j_0)}dt^1\dots dt^n$$
where $h_{i,j}=\partial_i\partial_jf(t_0)$ is the Hessian matrix and $\partial=\frac{\partial}{\partial t^i}$. We then
obtain
$$\int_{\Gamma}e^{f(t)/z}g(t)dt^1\dots dt^n \sim (-2\pi z)^{n/2}e^{f(t_0)/z}\sum_{k=0}^\infty c_k z^k\,\,\,\text{as}\,\,z\rightarrow 0$$
with
\begin{align}\label{AAsym}
    \sum_{k=0}^\infty c_kz^k=\frac{1}{\sqrt{\text{det}(h_{i,j})}}\left[e^{-\frac{z}{2}\sum_{i,j}h^{i,j}\partial_i\partial_j}e^{f_{\ge 3}/z}g(t)\right]_{t=t_0}
\end{align}
where $f_{\ge 3}(t)=f(t)-f(t_0)-\frac{1}{2}\sum_{i,j}h_{i,j}(t-t^i_0)(t^j-t^j_0)$ and $(h^{i,j})$ are the coefficients of the matrix inverse to $(h_{i,j})$.

\begin{Def}
 For a non-degenerate critical point $t_0$ of $f(t)$, we define the {\em formal asymptotic expansion}
 $$ \text{\em Asym}_{t_0}(e^{f(t)/z}g(t)dt)\in\CC[[z]] $$
to be the right-hand side of \eqref{AAsym}. Since the definition only involves the Taylor expansion at $t_0$, this is well-defined for germs $f(t),\,g(t)$ at $t_0$.
\end{Def}

\subsection{Givental's equivariant mirror for $K\PP^2$.}
The equivariant mirror for local $\PP^2$ was introduced by Givental, which is given by the (multivalued) Landau-Ginzburg potential
$$F=w_0+w_1+w_2+w_3+\sum_{i=0}^2 \lambda_i \text{log} w_i$$
defined on the family of affine varieties
$$Y_q=\{(w_0,w_1,w_2,w_3)\in\CC^4 : w_0w_1w_2=qw_3^3\}\,.$$

The associated oscillatory integral is of the form
\begin{align}\label{OI}
    \mathcal{I}=\int_{\Gamma\subset Y_q} e^{F/z}g(w) \omega
\end{align}
where $\omega$ is the (meromorphic) volume form on $Y_q$:
$$\omega=\frac{d \, \text{log} \,w_0 \wedge d\, \text{log}\,w_1 \wedge d\,\text{log} \,w_2 \wedge d \,\text{log}\,w_3}{d\, \text{log}\, q}\,.$$

Using the coordinate system $(w_0,w_1,w_2)$ on $Y_q$, we have

\begin{align*}
    \mathcal{I}=\int_{\Gamma\subset (\CC^*)^3}e^{(w_0+w_1+w_2+q^{-1/3}(w_0w_1w_2)^{1/3}+\sum_{i=0}^2\lambda_i \text{log}\,w_i)/z}g(w)\frac{1}{3}\frac{dw_0dw_1dw_2}{w_0w_1w_2}\,.
\end{align*}

\subsection{Formal asymptotic expansion.}
The proof of Lemma \ref{CI} is based on the computation of the formal asymptotic expansion of the integral $\mathcal{I}$ for $g(w)=1,w_3,w_3^2$.

We will use the specialization
$$\lambda_i=\zeta^i\,,$$
where $\zeta$ is the primitive third root of unity. With this specialization, the critical points of $F_{\lambda}$ are easy to calculate:
\begin{align}\label{cp}
    w_i=\left\{\begin{array}{rl}  
     L^{K\PP^2}-\zeta^i & \text{for}\,\,0\le i\le 2\, ,  \\
     -3L^{K\PP^2}       & \text{for}\,\,i=3\, ,   \end{array}\right.
\end{align}
where $L^{K\PP^2}=(1+27q)^{-\frac{1}{3}}$ as before. The three choices for the branch of $L^{K\PP^2}$ give rise to three critical points. For the sake of clarity, let us assume $q>0$ and choose the critical point corresponding to a real positive $L^{K\PP^2}$ in the following discussion. The critical value is given by
\begin{align*}
    F_\lambda(\text{cr})=\sum_{i=0}^2 \zeta^i\text{log}\,(L^{K\PP^2}-\zeta^i)
\end{align*}
where cr means the critical point \eqref{cp}. It can be decomposed as
\begin{align*}
    F_\lambda(\text{cr})=\text{log}\,(-9q)+\zeta\text{log}\,(1-\zeta)+\zeta^2\text{log}\,(1-\zeta^2)+\mu
\end{align*}
with $$\mu = \int_0^q (L^{K\PP^2}-1)\frac{dq}{q}\in q \CC[[q]]\,.$$The Hessian of $F_\lambda$ at the critical point with respect to logarithmic coordinates $(\text{log}\,w_0,\text{log}\,w_1,\text{log}\,w_2)$ is given by

$$\text{det}\left(\frac{\partial^2 F_\lambda(\text{cr})}{\partial \text{log}\,w_i\partial\text{log}\,w_j}\right)_{0\le i,j\le2}=-1\,.$$

The $I$-function of $K\PP^2$ was defined in \cite{LP} to be $H^*_{\mathsf{T}}(\PP^2)$-valued power series:
\begin{align*}
    I^{K\PP^2}(q,z)=\sum^{\infty}_{d=0}q^d\frac{\prod_{k=0}^{3d-1}(-3H-kz)}{\prod_{i=0}^2\prod_{k=1}^d(H-\lambda_i+kz)}\,.
\end{align*}

In \cite[Proposition 6.9]{CCIT2}, a relationship between the formal asymptotic expansion of the mirror oscillatory integral \eqref{OI} and the equivariant $I$-function was established for toric Deligne-Mumford stacks. Applying the result to $K\PP^2$, we obtain:
\begin{Prop}\label{AI}
We have
\begin{align*}
    e^{\mu/z} \text{\em Asym}_{\text{\em cr}}(e^{\frac{F_\lambda}{z}}\omega)=&I(q,z)|_{p_0}\cdot\frac{1}{3\sqrt{-1}}\text{\em Exp}\left(-\sum_{k=1}^\infty \frac{B_{k+1}(0)}{k(k+1)}N_{k,0} z^k\right)\,,\\
    e^{\mu/z} \text{\em Asym}_{\text{\em cr}}(e^{\frac{F_\lambda}{z}}w_3\omega)=&(-3z\mathsf{D}^{K\PP^2}-3H)I(q,z)|_{p_0}\\
    &\cdot\frac{1}{3\sqrt{-1}}\text{\em Exp}\left(-\sum_{k=1}^\infty \frac{B_{k+1}(0)}{k(k+1)}N_{k,0} z^k\right)\,,\\
    e^{\mu/z} \text{\em Asym}_{\text{\em cr}}(e^{\frac{F_\lambda}{z}}w_3^2\omega)=&(3z\mathsf{D}^{K\PP^2}+3H+z)(3z\mathsf{D}^{K\PP^2}+3H)I(q,z)|_{p_0}\\
    &\cdot\frac{1}{3\sqrt{-1}}\text{\em Exp}\left(-\sum_{k=1}^\infty \frac{B_{k+1}(0)}{k(k+1)}N_{k,0} z^k\right)\,, 
\end{align*}
where the $I$-function in the right-hand side should be expanded in Laurent series at $z=0$.
\end{Prop}

\noindent Here $\mathsf{D}^{K\PP^2}=q\frac{d}{dq}$. The definition of $\widetilde{P}_{i0}^{k,K\PP^2}$ immediately yields: 
\begin{Cor}\label{Asm}
We have
\begin{align*}
    3\sqrt{-1} \text{\em Asym}_{\text{\em cr}}(e^{\frac{F_\lambda}{z}}\omega)=\left(\sum_{k=0}^\infty \widetilde{P}_{00}^{k,K\PP^2} z^k\right) \text{\em Exp}\left(-\sum_{k=1}^\infty\frac{B_{k+1}(0)}{k(k+1)}N_{k,0} z^k\right)\,,\\
    \frac{-\sqrt{-1}}{L} \text{\em Asym}_{\text{\em cr}}(e^{\frac{F_\lambda}{z}}w_3\omega)=\left(\sum_{k=0}^\infty \widetilde{P}_{20}^{k,K\PP^2} z^k\right) \text{\em Exp}\left(-\sum_{k=1}^\infty\frac{B_{k+1}(0)}{k(k+1)}N_{k,0} z^k\right)\,,\\
    \frac{\sqrt{-1}}{3L^2} \text{\em Asym}_{\text{\em cr}}(e^{\frac{F_\lambda}{z}}w_3^2\omega)=\left(\sum_{k=0}^\infty \widetilde{P}_{10}^{k,K\PP^2} z^k\right) \text{\em Exp}\left(-\sum_{k=1}^\infty\frac{B_{k+1}(0)}{k(k+1)}N_{k,0} z^k\right)\,.
\end{align*}
\end{Cor}

\begin{proof}
In \cite{LP}, the evaluation of $\overline{\mathds{S}}_j(H^i)$ was obtained from $I^{K\PP^2}$ via Birkhoff factorization:
\begin{align}\label{SKP2}
    \nonumber\overline{\mathds{S}}_j(1)&= I^{K\PP^2}|_{p_j}\,,\\
    \overline{\mathds{S}}_j(H)&=\frac{(H+z\mathsf{D}^{K\PP^2})\overline{\mathds{S}}_j(1)}{C^{K\PP^2}_1}\,,\\
    \nonumber\overline{\mathds{S}}_j(H^2)&=\frac{(H+z\mathsf{D}^{K\PP^2})\overline{\mathds{S}}_j(H)}{C^{K\PP^2}_2}\,.
\end{align}
First two equations in the Corollary follow immediately from Proposition \ref{AI} using \eqref{ASKP2} and \eqref{SKP2}. The last equation in the Corollary requires further explanation.
\begin{eqnarray}\label{KP2H2}
    C_2^{K\PP^2}\overline{\mathds{S}}_j(H^2)&=&(H+z\mathsf{D}^{K\PP^2})\overline{\mathds{S}}(H)\\
    \nonumber
    &=&(H+z\mathsf{D}^{K\PP^2})(C_1^{K\PP^2})^{-1}(H+z\mathsf{D}^{K\PP^2})I^{K\PP^2}|_{p_j}\,.
\end{eqnarray}
In particular,
$$C_2^{K\PP^2}\overline{\mathds{S}}_j(H^2)=-z\frac{\mathsf{D}^{K\PP^2} C_1^{K\PP^2}}{(C_1^{K\PP^2})^2}(H+z\mathsf{D}^{K\PP^2})I^{K\PP^2}+\frac{1}{C_1^{K\PP^2}}(H+z\mathsf{D}^{K\PP^2})^2I^{K\PP^2}\Big|_{p_j}\, .$$
The analytic continuation of the hypergeometric series

$$C_1^{K\PP^2}=\sum_{d=0}^\infty\frac{(3d)!}{(d!)^3}(-q)^d$$
gives (see Appendix A of \cite{Ho})

$$C_1^{K\PP^2}=\frac{1}{3}\sum_{d=0}^\infty\frac{(-1)^d}{d!}\frac{\Gamma(\frac{1}{3}+\frac{n}{3})}{\Gamma(\frac{2}{3}-\frac{n}{3})^2}q^{-\frac{1}{3}-\frac{n}{3}}=\frac{1}{3}\frac{\Gamma(\frac{1}{3})}{\Gamma(\frac{2}{3})^2}q^{-\frac{1}{3}}+\mathcal{O}(q^{-\frac{1}{3}})\,\,\,\,\text{as}\,\,q\rightarrow \infty\,.$$

Since we are interested in the asymptotic expansion of $C_2^{K\PP^2}\overline{\mathds{S}}_j(H^2)$ at $q=\infty$, we can replace
$$\frac{\mathsf{D}^{K\PP^2}C_1^{K\PP2}}{C_1^{K\PP2}}\,\,\,\text{with}\,\,-\frac{1}{3}$$
and consider instead of \eqref{KP2H2} the function
$$\frac{1}{C_1^{K\PP^2}}(H+z\mathsf{D}^{K\PP^2}+\frac{z}{3})(H+z\mathsf{D}^{K\PP^2})I^{K\PP^2}\Big|_{p_j}\,.$$
The last equation of Corollary then follows from Proposition \ref{AI} using \eqref{ASKP2}.
\end{proof}

\subsection{Proof of Lemma \ref{CI}.}

We will obtain the identities of Lemma \ref{CI} by computing the analytic continuation of $$\text{Asym}_{\text{cr}}(e^{\frac{F_\lambda}{z}} g\omega)\ \ \text{as}\ q\rightarrow \infty\, .$$ Note that $L=0$ at $q=\infty$. The Landau-Ginzburg potential $F_\lambda$ near $q=\infty$ is mirror to $[\CC^3/\ZZ_3]$. Thus, $\text{Asym}_{\text{cr}}(e^{\frac{F_\lambda}{z}} g\omega)$ near $q=\infty$ can be computed in terms of the equivariant $I$-function of $[\CC^3/\ZZ^3]$, again by \cite[Proposition 6.9]{CCIT2}.

It is instructive to evaluate directly the oscillatory integral at $q=\infty$. The key ingredient is the Stirling approximation for the $\Gamma$-function:
$$\text{log}\, \Gamma(h+x) \sim \left(x+h-\frac{1}{2}\right) \text{log}\,x-x+\frac{1}{2}\text{log}\,(2\pi)+\sum_{k=1}^\infty \frac{(-1)^{k+1}B_{k+1}(h)}{k(k+1)x^k}\,.$$
We have
\begin{eqnarray*}
    \int_\Gamma e^{F_\lambda/z}\omega \big|_{q=\infty}&=&\int_\Gamma e^{\sum_{i=0}^2(w_i+\lambda_i \text{log}\,w_i)/z}\frac{1}{3}\frac{dw_0dw_1dw_2}{w_0w_1w_2}
    \\
    &=&\frac{1}{3}\prod_{i=0}^2\Gamma\left(\frac{\lambda_i}{z}\right)(-z)^{\lambda_i/z}\,.
\end{eqnarray*}
Using the Stirling approximation, we obtain
\begin{align*}
    \text{Asym}_{\text{cr}}(e^{F_\lambda/z}\omega)\big|_{q=\infty}&=\frac{1}{3\sqrt{-\lambda_0\lambda_1\lambda_2}}\text{Exp}\,\left(\sum_{i=0}^2\sum_{k=1}^\infty\frac{(-1)^{k+1}B_k(0)}{k(k+1)\lambda_i^k}z^k\right)\\
    &=\frac{1}{3\sqrt{-1}}\text{Exp}\,\left(3\sum_{k=1}^\infty \frac{(-1)^{3k+1}B_{3k}(0)}{3k(3k+1)}z^{3k}\right)\,.
\end{align*}
This, together with Corollary \ref{Asm}, immediately gives the first identity in Lemma \ref{CI}. We also have
\begin{align*}
    \int_{\Gamma}e^{F_\lambda/z}\frac{w_3}{L} \omega\big|_{q=0}&=\int_\Gamma e^{\sum_{i=0}^2(w_i+\lambda_i \text{log}\,w_i)/z}(w_0w_1w_2)^{1/3}\frac{dw_0dw_1dw_2}{w_0w_1w_2}\\
    &=-z\prod_{i=0}^2\Gamma\left(\frac{1}{3}+\frac{\lambda_i}{z}\right)(-z)^{\lambda_i/z}\,.
\end{align*}
Again by the Stirling approximation, we obtain
\begin{align*}
    \text{Asym}_{\text{cr}}\left(e^{F_\lambda/z}\frac{w_3}{L}\omega\right)\big|_{q=\infty}&=\frac{-(\lambda_0\lambda_1\lambda_2)^{1/3}}{\sqrt{-\lambda_0\lambda_1\lambda_2}}\text{Exp}\,\left(\sum_{i=0}^2\sum_{k=1}^\infty \frac{(-1)^{k+1}B_{k+1}(\frac{1}{3})}{k(k+1)\lambda_i^k}z^k\right)\\
    &=\frac{-1}{\sqrt{-1}}\text{Exp}\left(3\sum_{k=1}^\infty \frac{(-1)^{3k+1}B_{3k+1}(\frac{1}{3})}{3k(k+1)}z^{3k}\right)
\end{align*}
which, together with Corollary \ref{Asm}, gives the second identity in Lemma \ref{CI}. Similarly, we also have
\begin{align*}
    \text{Asym}_{\text{cr}}\left(e^{F_\lambda/z}\frac{w_3^2}{L^2}\omega\right)\big|_{q=\infty}
    &=\frac{3}{\sqrt{-1}}\text{Exp}\left(3\sum_{k=1}^\infty \frac{(-1)^{3k+1}B_{3k+1}(\frac{2}{3})}{3k(k+1)}z^{3k}\right)
\end{align*}
which, together with Corollary \ref{Asm}, gives the third identity in Lemma \ref{CI}. \qed

\end{document}